\input ./style/arxiv-general.cfg
\documentclass[MSNbibl,number,citesort,seceqn,dvips]{arxbj}
\makeatletter
   \@ifpackageloaded{graphicx}{}{\usepackage{graphicx}}
\makeatother
\usepackage{graphicx}
\volume{23}
\issue{1}
\pubyear{2017}
\firstpage{288}
\lastpage{328}
\doi{10.3150/15-BEJ744}% Updated by VTEXPTS2LaTeX.exe, 22.10.2015 10:43
\docsubty{FLA}

\makeatletter
\def\sfrac#1#2{#1/#2}

\def\sklfrac#1#2{(#1/#2)}

\newcommand{\lleft}{\left}
\newcommand{\rrvert}{\vert}
\newcommand{\rright}{\right}
\newcommand{\rrVert}{\Vert}
\newcommand{\llvert}{\vert}
\newcommand{\llVert}{\Vert}
\newcommand{\eqref}[1]{(\ref{#1})}
\newcommand{\rank}{\operatorname{rank}}
\newcommand{\supp}{\operatorname{supp}}
\newcommand{\im}{\operatorname{Im}}
\newcommand{\sign}{\operatorname{sign}}
\newcommand{\PH}{\mathrm{PH}}
\newcommand{\Var}{\operatorname{Var}}
\def\Q{\mathbb{Q}}
\def\R{\mathbb{R}}
\def\F{\mathbb{F}}
\def\Z{\mathbb{Z}}
\def\cD{\mathcal{D}}
\def\cM{\mathcal{M}}
\def\cN{\mathcal{N}}
\def\cP{\mathcal{P}}
\def\cS{\mathcal{S}}
\def\cX{\mathcal{X}}
\def\cY{\mathcal{Y}}
\newcommand{\E}{\mathbb{E}} %expectation
\newcommand{\bI}{{\mathbf{I}}}
\newcommand{\eps}{\varepsilon}
\newcommand{\be}{\mathbf{e}}
\newcommand{\given}{\mid}
\newtheorem{lem}{Lemma}[section]
\newtheorem{teo}[lem]{Theorem}
\newproclaim{defn}[lem]{Definition}
\newremark{rem}{Remark}
\newcommand{\iid}{\mathrm{i.i.d.}}
\newcommand{\ninf}{n\to\infty}
\newcommand{\fmax}{p_{\max}}
\newcommand{\fmin}{p_{\min}}
\def\Ymax{Y_{\max}}
\def\hPH{\widehat{\PH}}
\def\Dgm{\mathrm{Dgm}}
\def\Diag{\mathrm{Diag}}
\newcommand{\hookdrarrow}{\mathrel{\rotatebox[origin=c]{45}{$\hookrightarrow$}}}
\newcommand{\hookurarrow}{\mathrel{\rotatebox[origin=c]{-45}{$\hookrightarrow$}}}
\makeatother

\begin{document}
\begin{frontmatter}

\title{Topological consistency via kernel estimation}
\runtitle{Topological consistency}

\begin{aug}
%%%% inicialai - be tarpu
% Corresponding author: Omer Bobrowski - omer@math.duke.edu% Updated by
%VTEXPTS2LaTeX.exe, 22.10.2015 10:43
\author[A]{\inits{O.}\fnms{Omer}~\snm{Bobrowski}\corref
{}\thanksref{A}\ead[label=e1]{omer@math.duke.edu}},
\author[B]{\inits{S.}\fnms{Sayan}~\snm{Mukherjee}\thanksref{B}\ead
[label=e2]{sayan@stat.duke.edu}}
\and
\author[C]{\inits{J.E.}\fnms{Jonathan E.}~\snm{Taylor}\thanksref
{C}\ead[label=e3]{jonathan.taylor@stanford.edu}}
%%\runauthor{} %% auto
%\dedicated{}
\address[A]{Department of Mathematics, Duke University, Durham NC
27708, USA.\\ \printead{e1}}
\address[B]{Departments of Statistical Science, Computer Science, and
Mathematics, Duke University, Durham NC 27708, USA. \printead{e2}}
\address[C]{Department of Statistics, Stanford University, Stanford,
CA 94305-4065, USA.\\ \printead{e3}}
\end{aug}

% HISTORY:
%
\received{\smonth{12} \syear{2014}}% Updated by VTEXPTS2LaTeX.exe,
%22.10.2015 10:43
%
\revised{\smonth{4} \syear{2015}}% Updated by VTEXPTS2LaTeX.exe,
%22.10.2015 10:43

% ABSTRACT
%
\begin{abstract}
We introduce a consistent estimator for the homology (an algebraic
structure representing
connected components and cycles) of level sets of both density and
regression functions.
Our method is based on kernel estimation. We apply
this procedure to two problems: (1)~inferring the homology structure
of manifolds from noisy observations, (2)~inferring the persistent homology
(a~multi-scale extension of homology) of either density or regression functions.
We prove consistency for both of these problems. In addition to the theoretical
results, we demonstrate these methods on simulated data for binary
regression and clustering applications.
\end{abstract}

% KEYWORDS
% visi is mazosios raides ir pagal abecele
%
\end{frontmatter}

%s1 #&#
\section{Introduction} \label{Introduction}

Level set estimation for probability density functions has been
extensively studied in the past few decades.
The basic formulation of the problem is as follows. Let $p:\R^d \to\R
$ be an unknown probability density function and define $D_L:=  \{
x\in\R^d: p(x)\ge L \}$ to be the $L$th super level set of $p$ (from
here on we will drop the word ``super''). Given a sample $ \{
X_1,\ldots, X_n \}$ of $\iid$ observations drawn from $p$, we would
like to estimate the set $D_L$. Recovering the level sets of density
functions have shown to be useful in various applications such as
clustering and cluster analysis \cite
{cuevasestimating2000,cuevascluster2001,hartiganclustering1975,mullerexcess1992,mullerexcess1991,polonikmeasuring1995},
pattern recognition \cite
{cuevaspattern1990,devroyedetection1980,grenanderabstract1981},
anomaly detection \cite{bailloset2000,cuevasplug-approach1997},
and econometrics \cite
{deprinsfarrell1983,deprinsmeasuring2006,farrellmeasurement1957,korostelevminimax1993}
(where recovering the support of a distribution and its boundary is
used for measuring efficiency).

Various solutions have been proposed to the level set estimation
problem. Standard solutions include the plug-in estimator \cite
{bailloconvergence2001,baillototal2003,cuevasplug-approach1997,molchanovempirical1991,molchanovlimit1998},
the excess mass estimator \cite
{hartiganestimation1987,mullerexcess1992,mullerexcess1991,nolanexcess-mass1991,polonikmeasuring1995,tsybakovnonparametric1997},
and the ``naive'' estimator \cite
{cuevasboundary2004,devroyedetection1980,walthergranulometric1997}.
The distance measure used to evaluate the performance of these
estimators is usually either the Hausdorff distance or the Lebesgue
distance (the volume of the difference between two sets). In this paper
we wish to study level sets estimation from a topological perspective.
Rather than trying to achieve an accurate recovery for the actual shape
of the level sets, we wish to recover their qualitative topological
properties (such as connected components and holes). Unfortunately,
minimizing the Hausdorff or Lebesgue distance does not provide any
guarantees for the quality of the topological recovery. Therefore, we
have to consider a new type of an estimator. The sets in Figure~\ref
{fig:level_est} demonstrate the fact that minimizing the Hausdorff (or
Lebesgue) distance can still result in very different topological spaces.

The motivation for studying the topology of level sets comes from the
clustering problem. Given a set of observations generated by a
probability density function $p:\R^d\to\R$, clustering can be
loosely described as identifying and characterizing the connected
components of either the support of $p$ or one of its level sets (cf.
\cite{hartiganclustering1975,MacQueen67,SneathSokal73,Tyron39}).
From a topological perspective, clustering can be viewed as a question
about the \emph{homology} of the level sets. Briefly, the homology of
a topological space $X$ is a set of Abelian groups, denoted by $ \{
H_0(X), H_1(X),\ldots \}$, where the elements of $H_0(X)$ contain
information about the connected components of $X$, and for $k>0$, the
group elements of $H_k(X)$ contain information about ``cycles,'' or
``holes'' of different dimensions (see Section~\ref{sec:topdefs} for
more details).
From the perspective of algebraic topology, the clustering problem is
thus equivalent to recovering $H_0(X)$ where $X$ is either the support
of the distribution or a selected level set.
A statistical perspective of the recent efforts in topological data
analysis (TDA) \cite
{balakrishnanminimax2012,carlssontopology2009,edelsbrunnerpersistent2008,niyogifinding2008,niyogitopological2011}
has been to extract topological invariants, and homology in particular,
from random data. For example, recovering $H_1$ provides information
about holes or loops in the data, which is useful in various
applications such as network coverage \cite{desilvacoverage2007} or
recovering periodic behavior \cite{pereasliding2013}. The idea is
that these topological summaries are useful for statistical inference
and robust under various transformations. Our goal is therefore to
examine level set estimation when the objective is not only to recover
$H_0(X)$ but rather the entire set of homology groups.

%%%%%%%%%%

%
%f1 #&#
\begin{figure}[t]

\includegraphics{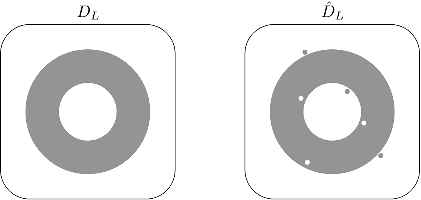}

\caption{A schematic picture illustrating the difficulty in estimating
the homology of level sets. Suppose that $D_L$ is the annulus on the
left and $\hat D_L$ is its estimate on the right. While in both
Hausdorff and Lebesgue distance the sets $D_L$ and $\hat D_L$ are
close, the homology of these sets is completely different. In
particular, $D_L$ has a single connected component and a single hole,
while $\hat D_L$ has four of each. By taking the radius of the small
circles to be as small as we wish, we can make both the Hausdorff and
the Lebesgue distance to be arbitrarily small, while topologically we
are looking at two different spaces.}
\label{fig:level_est}
\end{figure}

%%%%%%%%%%

The idea of characterizing points or subsets of $\R^d$ by their
homology was developed in a series of papers in the late 1990s \cite
{robinscomputing1998,robinscomputing2000}. Asymptotic and
non-asymptotic analysis of consistency and convergence of
topological summaries as the number of observations increase has been
examined for a variety of geometric objects using a variety of
statistical and probabilistic tools
\cite
{adlerpersistent2010,adlercrackle2014,balakrishnanminimax2012,balakrishnantight2013,bendichlocal2012,bobrowskidistance2014,bobrowskitopology2014,bubenikstatistical2010,chazalsampling2009,kahlerandom2011,kahlelimit2013,niyogifinding2008,niyogitopological2011}.
In the statistics and empirical process community, a version of the
topology inference problem was presented as inference of the empirical
geometry of data \cite{Koltchinskii2000}.

The main objective of this paper is to provide a consistent method for
recovering the homology of the level sets $D_L$ of functions $f:\R
^d\to\R$,
where $f$ will be either a probability density function or a regression
function. The standard plug-in idea would be to use a kernel-based
estimator $\hat f$ to construct an estimator $\hat D_L$ to the level
set. The problem with this approach is that due to the discrete nature
of homology even a tiny error in the set estimate $\hat D_L$ can
introduce a significant error in homology. For example, an
infinitesimally small region included by mistake can increase the
number of components, while a small region excluded by mistake might
introduce a hole. Such errors in homology estimation may occur no
matter how small the extraneous components and holes are. This problem
is illustrated in Figure~\ref{fig:level_est}.

%
%f2 #&#
\begin{figure}[b]

\includegraphics{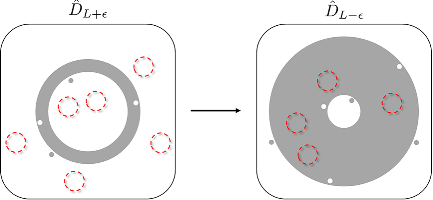}

\caption{An\vspace*{1pt} illustration of the filtering mechanism underlying the
homology estimator presented in this paper. Suppose that the set of
interest $D_L$ is the same as in Figure~\protect\ref{fig:level_est}.
Both estimates $\hat D_{L-\eps}$ and $\hat D_{L+\eps}$ have the wrong
homology. The dashed circles in each figure mark the locations of the
extraneous features (components and holes) in the other. We observe
that none of the extraneous features exist in both sets. Since the
image of the map $\imath_*$ contains only the topological features
that exist in both $\hat D_{L+\eps}$ and $\hat D_{L-\eps}$, it will
consist of a single component and a single hole -- the correct homology
of $D_L$.}\label{fig:levels_sketch}
\end{figure}

The main result in this paper presents a robust homology estimator for
the level sets of both density and regression functions, that overcomes
these difficulties. We show that instead of using $\hat D_L$ as an
estimate, one should consider the inclusion map between the nested
pairs~-- $\hat D_{L+\eps} \subset\hat D_{L-\eps}$ (for a properly
chosen $\eps>0$). The key object of interest is then the following
induced map between
the homology groups of the two level sets:
\[
\imath_*: H_*(\hat D_{L+\eps})\to H_*(\hat D_{L-\eps}),
\]
where ``$*$'' is a standard notation for an arbitrary degree. Inference
of the homology at a single level is
noisy, however the map $\imath_*$ serves as a filter for the
homological noise (see Figure~\ref{fig:levels_sketch}). In particular,
we will show that the image of this map -- $\im(\imath_*)$ -- is
isomorphic to the homology of $D_L$ with a high probability. This
statement is formalized by Theorem~\ref{teo:persistence}.

There are two direct implications for recovering the homology of level
sets: recovering the homology of a manifold from a noisy sample and
inference of the persistent homology of a function. For both
applications, we make use of kernel density estimation to infer the
image of the map $\imath_*$ between the homology groups of different level
sets. An interesting observation is that the conditions to recover the
homology of the manifold or regression function do not require
consistency of the kernel estimator.

The first application is inferring the homology of a manifold from a
noisy sample. This problem was previously studied in \cite
{balakrishnanminimax2012,niyogitopological2011}.
In this paper, we show that for a wide class of noise models one can
recover the homology of a manifold using fewer assumptions than
previous methods and analysis. This result is stated in Theorem~\ref
{teo:manires}.

%%%%%%%%%%

%%%%%%%%%%

The second application is estimating the \emph{persistent homology} of
the function $f$. Persistent homology (described in Section~\ref
{sec:topdefs}) is a multi-scale topological summary. The main idea is
instead of considering the homology of a single level $D_L$, the entire
sequence of level sets is considered as $L$ decreases from $\infty$ to
$-\infty$. One then tracks at what values of $L$ changes in homology
occur. The logic behind this computation is that homological features
that persist across a wide range of levels are stable features while
the other homological features are transient or noisy. This result is
stated in Theorem~\ref{teo:ph_consistency}.

The paper is structured as follows. In Section~\ref{sec:topdefs}, we
state the topological concepts and definitions we will use in this
paper, namely homology and persistent homology. The main results of the
paper are stated in Section~\ref{sec:results} with the proofs in the
\hyperref[append]{Appendix}. In Section~\ref{sec:compute}, we provide a procedure to
estimate the
homology of level sets. Intuition about the estimator as well as
results on simulated data are given in Section~\ref{simulations}. We
close with a discussion.

%%%%%%%%%%

%s2 #&#
\section{Topological preliminaries} \label{sec:topdefs}

%%%%%%%%%%

In this section, we introduce the basic ideas of homology and
persistent homology. To help fix ideas, we first present a particular
example of persistent homology related to agglomerative hierarchical clustering.

%%%%%%%%%%

%s2.1 #&#
\subsection{Homology}

%%%%%%%%%%

We develop the concept of homology intuitively, for a more rigorous and
comprehensive treatment see \cite
{hatcheralgebraic2002,munkreselements1984}.
Let $X$ be a topological space. The \textit{homology} of $X$ is a set
of Abelian groups $ \{H_k(X) \}_{k=0}^\infty$, called
homology groups.
In this paper, we consider homology with coefficients in a field
$\mathbb{F}$, in this case $H_k(X)$ is actually a vector space. The
zeroth homology group $H_0(X)$ is generated by elements that represent
connected components of $X$. For example, if $X$ has three connected
components, then $H_0(X) \cong\mathbb{F} \oplus\mathbb{F} \oplus
\mathbb{F}$ (here $\cong$ denotes group isomorphism), and each of the
three generators of this group corresponds to a different connected
component of $X$. For $k\ge1$, the $k$th homology group $H_k(X)$ is
generated by elements representing $k$-dimensional ``holes'' or
``cycles'' in $X$. An intuitive way to think about a $k$-dimensional
hole is as the result of taking the boundary of a $(k+1)$-dimensional
body. For example, if $X$ is a circle then $H_1(X) \cong\mathbb{F}$,
if $X$ is a 2-dimensional sphere then $H_2(X) \cong\mathbb{F}$, and
in general if $X$ is a $n$-dimensional sphere, then
\[
H_k(X) \cong\cases{ \mathbb{F}, &\quad$k=0,n$,
\cr
\{0 \}, &\quad
otherwise.}
\]
Another interesting example is the $2$-dimensional torus denoted by $T$
(see Figure~\ref{fig:torus}). The torus has a single connected
component, therefore $H_0(T) \cong\mathbb{F}$, and a single
$2$-dimensional hole (the void inside the surface) implying that
$H_2(T) \cong\F$. As for $1$-cycles (or closed loops) the torus has
two distinct features (see Figure~\ref{fig:torus}) and therefore
$H_1(T) \cong\F\oplus\F$.

The ranks of the homology groups (the number of generators) are called
the Betti numbers, and are denoted by $\beta_k(X) \triangleq\rank
(H_k(X))$. When we refer to all the homology groups simultaneously, we
use the notation $H_*(X)$.

%%%%%%%%%%

%
%f3 #&#
\begin{figure}[t]

\includegraphics{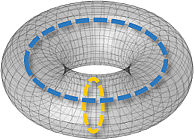}

\caption{The $2$-dimensional torus and its cycles. The torus has a
single connected component and a single $2$-cycle (the void locked
inside the torus). In addition, it has two distinct $1$-dimensional
cycles (or closed loops) represented by the two curves in the figure.
Consequently, the Betti numbers of the torus are $\beta_0 = 1, \beta
_1 = 2, \beta_2 =1 $.}
\label{fig:torus}
\end{figure}

%%%%%%%%%%

In addition to providing a summary for a single space, homology can
also characterize the topological behavior of functions. Let $f:X\to Y$
be a map between two topological spaces, then homology theory provides
a way to define the ``induced map'' $f_*: H_*(X)\to H_*(Y)$ mapping
between the homology groups of the two spaces.

Another term we will use is \emph{homotopy equivalence} (cf. \cite
{hatcheralgebraic2002,munkreselements1984}). Loosely speaking, two
topological spaces $X,Y$ are homotopy equivalent if we can continuously
transform one into the other. We denote this property by $X\simeq Y$.
If $X\simeq Y$ then they have the same homology, that is, $H_*(X) \cong H_*(Y)$.

%%%%%%%%%%

%s2.2 #&#
\subsection{Persistent homology}\label{sec:intro_ph}

%%%%%%%%%%

Let $\cX=  \{X_t \}_{t=a}^b$ be a filtration of topological spaces,
such that $X_{t_1} \subset X_{t_2}$ if $t_1 < t_2$.
As the parameter $t$ increases, the homology of the spaces $X_t$ may
change (e.g., components are added and merged, cycles are formed and
filled up). The \emph{persistent homology} of $\cX$, denoted by $\PH
_*(\mathcal{X})$, keeps track of this process. Briefly, $\PH_*(\cX)$
contains the information about the homology of the individual spaces $\{
X_t\}$ as well as the mappings between the homology of $X_{t_{1}}$ and
$X_{t_{2}}$ for every $t_{1} < t_{2}$. The \emph{birth time}
of an element in $\PH_*(\cX)$ can be thought of as the value of $t$
where this element appears for the first time.
The \emph{death time} is the value of $t$ where an element vanishes,
or merges with another existing element. We refer the reader to \cite
{edelsbrunnerpersistent2008,edelsbrunnercomputational2010,ghristbarcodes2008,zomorodiancomputing2005}
for more details and formal definitions. Another perspective of
persistence homology is as a summary statistic of point cloud data that
is robust to certain invariances, this perspective has been developed
in \cite
{blumbergrobust2013,bubenikstatistical2015,mileykoprobability2011,turnerfrechet2012}.

A useful way to describe persistent homology is via
the notion of \textit{barcodes}.
A barcode for the persistent homology of a filtration $\cX$
is a collection of graphs, one for each order of homology group. A
bar in the $k$th graph, starting at $b$ and ending at $d$ ($b\le d$)
indicates the existence of a generator of $H_k(X_t)$ (or a $k$-cycle)
whose birth and death times are $b,d$, respectively.
In Figure~\ref{fig:ph_example}, we present an example for a barcode
generated in the following way. We take a sample of $n=50$ points
$P_1,\ldots, P_n \in\R^2$ sampled from a uniform distribution on an
annulus. We then define $X_r = \bigcup_{i} B_r(P_i)$ to be the union
of closed balls around the sample points. Increasing $r$ makes the
space $X_r$ grow. In this process connected components merge, and
cycles are formed and then filled up. In Figure~\ref
{fig:ph_example}(a), we present a few snapshots of the space $X_r$ for
different values of $r$ where different features show. The barcode in
Figure~\ref{fig:ph_example}(b) presents a summary of all the homology
features in this process. We can see that there are two bars that are
significantly longer than the others (one in $H_0$ and one in $H_1$)
indicating that the underlying space has a single connected component,
and a single cycle (as the annulus does).

%%%%%%%%%%

%
%f4 #&#
\begin{figure}

\includegraphics{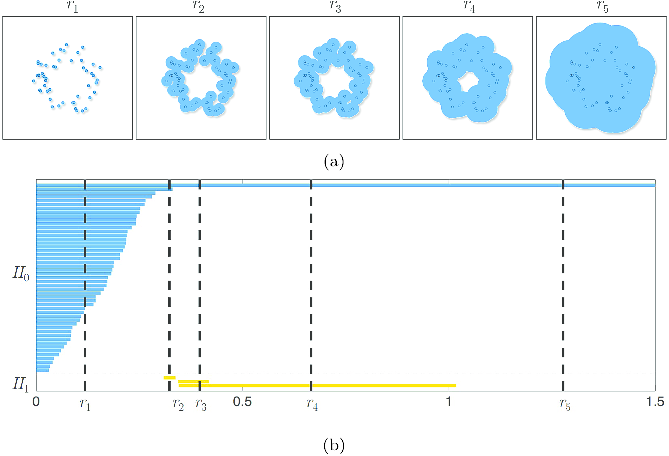}

\caption{(a) $X_r$ is a union of balls of radius $r$ around a random
set of $n=50$ points, generated from a uniform distribution on an
annulus in $\R^2$. We present five snapshots of this filtration.
(b) The persistent homology of the filtration $\{X_r\}_{r\ge0}$. The
$x$-axis is the radius of the balls, and the bars represent the
homology features that are born and died. For $H_0$ we observe that at
radius zero the number of components is exactly $n$ and as the radius
increases components merge (or die). Note that when two components
merge, we terminate the bar for one of them, and the merged component
is represented by the bar we keep. This is a standard representation
that comes as the result of the algebraic structure underlying
persistent homology (cf. \cite{zomorodiancomputing2005}). The cycles
show up later in this process. There are two bars that are
significantly longer than the others (one in $H_0$ and one in $H_1$).
These correspond to the true topological features of the annulus.}\label{fig:ph_example}
\end{figure}

%%%%%%%%%%

For a given space, there are many choices of filtrations (sequences of
nested subspaces). In this paper the filtrations we work with are the
(super) level sets of functions.
Specifically, let $f:\R^d\to\R$ and let $D_L$ be a level set of $f$.
As the level $L$ is decreased from $\infty$ to $-\infty$ the sets
$D_L$ grow, and in this process components and cycles are created and
destroyed. We denote by $\PH_*(f)$ the persistent homology for this process.

To show later that we can recover the persistent homology structure, we
will need a notion of distance between the persistent homology of two
different filtrations.
If $\cX$ is a filtration, the \emph{$k$th persistence diagram} of
$\cX$, denoted by $\Dgm_k(\cX)$ is the set of all pairs $(b,d)$ of
birth--death times of features in $\PH_k(\cX)$.
The bottleneck distance between the persistent homology of the
two filtrations $\cX$ and $\cY$ is defined as
\[
d_B\bigl(\PH_k(\cX),\PH_k(\cY)\bigr) =
\inf_{\gamma\in\Gamma}\sup_{p \in
{\Dgm_k}(\cX)} \bigl\llVert p-
\gamma(p)\bigr\rrVert _{\infty}.\
\]
The set $\Gamma$ consists of all the bijections $\gamma:{\Dgm_k(\cX
)\cup\Diag}\to{\Dgm_k(\cY)\cup\Diag}$, where
$\Diag=  \{(x,x): x\in\R \}\subset\R^2$ is the
diagonal line,
and $\llVert \cdot\rrVert _\infty$ is the sup-norm in $\R^2$.
In other words, we are looking for a matching between the points in
$\Dgm_k(\cX)$ and $\Dgm_k(\cY)$ that requires the minimal
translations of birth and death times. We add the diagonal to each
diagram for two reasons. First, we want to be able to consider diagrams
with different numbers of features, and second, we want to allow
deleting points from a diagram (by matching them to the diagonal)
rather than forcing them to match.

%%%%%%%%%%

%
%f5 #&#
\begin{figure}

\includegraphics{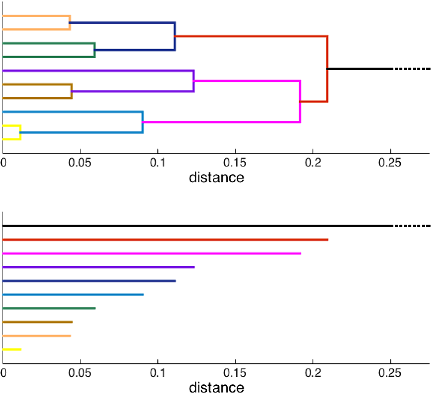}

\caption{Persistent homology and hierarchical clustering. The figure
on top is the dendrogram generated by a set of 10 random points in the
interval $[0,1]$. The bottom figure is the barcode generated by the
$0$-persistent homology for the sub-level sets of the distance function
from the same set of points. The $x$-axis represents function values
(distance, in our case). In this example, all the connected components
are created at distance zero, and only differ by their death point
(when two components merge). Note, that one of the components (the top
bar) lives forever. The death points in the barcode correspond to nodes
in the dendrogram, we marked the bars with different colors matching
the relevant part of the dendrogram.}\label{fig:dendro}\label{default}
\end{figure}

To conclude this section, we note that the zeroth persistent homology,
$\PH_0$, is closely related to hierarchical clustering as the
following example will illustrate.
Let $\cP\subset\R^d$ be a finite set of points in Euclidean space.
We define the distance function from the set $d_{\cP}:\R^d \to\R$ as
\[
d_{\cP}(x) = \min_{p\in\cP} \llVert x-p\rrVert.
\]
In this case, computing the $0$th persistent homology for the sub level
set filtration of $d_\cP$ is very simple. We start at level $0$ with
just the finite set $\cP$, and as we increase the level we merge connected
components according to the distances between points in $\cP$. The bottom of
Figure~\ref{fig:dendro} is the barcode generated by such a process,
the top figure is the dendrogram generated by the same set of points.
One can observe that the end points of the bars in the barcode are the
nodes in the dendrogram.%\looseness=1

%%%%%%%%%%

%s3 #&#
\section{Statistical model and main results}\label{sec:results}

%%%%%%%%%%

Given a function $f:\R^d\to\R$ the objects we analyze in this
paper are the (super) level sets of $f$
%
%
%e3.1 #&#
\begin{equation}
\label{eq:def_DL} D_L \triangleq \bigl\{x\in\R^d: f(x) \ge
L \bigr\}.
\end{equation}
Note that for any $L_1 < L_2$ we have $D_{L_2} \subset D_{L_1}$.

Previous results on level set estimation usually require some
assumptions on either the function $f$ (smooth, non-flat, etc.), or the
shape of the level set (convex, star-shaped, elliptic, etc.). For the
purpose of homology estimation, our main assumption on $f$ is
``tameness'' as defined in \cite{bubenikcategorification2014}.

%%%%%%%%%%

%
%de3.1 #&#
\begin{defn}
Let $f:\R^d \to\R$, and $D_L$ as defined in \eqref{eq:def_DL}.
\begin{enumerate}[2.]
\item[1.]
We say that $L$ is a \emph{homological regular value} if there exists
$\eps> 0$ such that
for every $ v_2 \le v_1 $ in $(L-\eps,L+\eps)$ the map
$H_k(D_{v_1})\to H_k(D_{v_2})$ induced by inclusion is an isomorphism
for every $k\ge0$.

Otherwise, we say that $L$ is a \emph{homological critical value}.

\item[2.] A function $f$ is called \emph{tame} if it has a finite number
of homological critical values, and $\rank(H_k(D_L))$ is finite for
all $L$ and $k$.
\end{enumerate}
\end{defn}

%%%%%%%%%%

Our main goal in this paper is to present a consistent method for
recovering the homology of a given level set $D_L$. We
will examine the level sets of two classical quantities of interest
in statistics:
\begin{enumerate}[2.]
\item[1.] Density functions~-- Given $\operatorname{Data} = \{X_1,\ldots, X_n \}
\stackrel{\iid}{\sim} p(x)$, where $p$ is a probability density
function, our objective is to recover the level sets of $f=p$.

\item[2.] Regression functions~-- Given $\operatorname{Data} = \{(X_1,Y_1),\ldots, (X_n,Y_n) \}
\stackrel{\iid}{\sim} p_{X,Y}(x,y)$, where $p_{{X,Y}}(x,y)$ is a joint
probability density function and we state $p:\R^d\to\R$ as the
marginal density of $X$.
Our objective is to recover the level sets of the regression function
$f(x) \triangleq\E \{Y\given X=x \}$.
\end{enumerate}
A common procedure to recover the homology of an unknown space $S$ from a
random sample $\cX\subset S$ is to compute the homology of a
union of closed balls around the sample points
%
%
%e3.2 #&#
\begin{equation}
\label{eq:union_balls} U(\cX, r):= \bigcup_{X\in\cX}
B_r(X),
\end{equation}
for some choice of radius $r$ (cf. \cite
{bobrowskitopology2014,niyogifinding2008}).
In the level-set estimation literature, this procedure is known as the
``naive'' estimator \cite
{cuevasboundary2004,devroyedetection1980,walthergranulometric1997}.
We can use this idea to estimate the homology of the set $D_L$ using
the following procedure (P1):
\begin{enumerate}
\item Use the entire data set to construct an estimator $\hat f$.
\item Using the estimator $\hat f$, define
\[
\cX^L = \bigl\{ X_i: \hat f(X_i) \ge L\bigr
\},
\]
as the set of data points lying in the $L$th level set of $\hat f$.
\item Consider $U(\cX^L, r)$ as an
estimate of $D_L$, and the homology $H_*(U(\cX^L, r))$ as an
estimate of $H_*(D_L)$.
\end{enumerate}
We will use kernel estimators for $\hat f$ in both the
regression and density estimation case. A key difficulty in the above
procedure is that the estimator $\hat f$ may introduce errors in the\vspace*{1pt}
filtering step 2 of the above procedure. In \cite
{cuevasboundary2004,devroyedetection1980,walthergranulometric1997}
it is shown that small errors in the estimate $\hat f$ are translated
to small errors in terms of the Hausdorff or Lebesgue distances.
However, since homology is a discrete descriptor, even tiny errors in the
filtering step can introduce large errors in the homology estimates.
For example, even a single point incorrectly included in the
level set assignment can form an extra connected component, and
increase the zeroth Betti number by one (see Figure~\ref{fig:level_est}).
One of the main challenges we will address in this paper is providing
an estimator that is robust to this type of error.

Given a kernel function $K:\R^d\to\R$ we construct our estimators as
follows. In the density estimation case,
we define
\[
\hat f_n(x) = \hat p_n(x) \triangleq\frac{1}{n \times C_K r^d}
\sum_{i=1}^n K_r(x-X_i),
\]
where $X_1,\ldots,X_n$ are the observed data, $K_r(x) = K(x/r)$, and
$C_K$ is a normalizing constant defined below. In the regression
setting, we use the Nadaraya--Watson estimator \cite
{Nadaraya64,silvermandensity1986,Watson64}
\[
\hat f_n(x) \triangleq\frac{\sum_{i=1}^n Y_i K_r(x-X_i)}{\sum_{i=1}^n K_r(x-X_i)},
\]
where $\{(X_1,Y_1),\ldots,(X_n,Y_n)\}$ are the observed data.

The kernel functions $K(x)$ we consider satisfy the following
conditions (C1):
\begin{enumerate}
\item The support of the kernel function is contained within the unit
ball of radius $1$, that is, $\supp(K) \subset B_1(0)$.
\item The kernel function has a maximum at the origin, with $K(0) = 1$,
and $\forall x: K(x) \in[0,1]$.
\item The kernel function is smooth within the unit ball, and
\[
\int_{\R^d} K(\xi)\,d\xi= C_K\qquad\mbox{for }
C_K \in(0,1).
\]
\end{enumerate}

Note that the bounded support assumption is very common in level set
estimation procedures (e.g., \cite
{bailloconvergence2001,cuevasplug-approach1997,walthergranulometric1997}).
Weak regularity conditions on the density or regression function will
be required to prove consistency of
the estimates of the homology of level sets. For both density
estimation and regression, we require the density function $p$ to be
tame and bounded, and we define
\[
\fmax\triangleq\sup_{x\in\R^d} p(x).
\]
For density estimation, we also require that for every $L$ the set
$D_L\subset\R^d$ is bounded.
For the regression case, we require in addition the following set of
conditions (C2):
\begin{enumerate}
\item The marginal density of $X$ has compact support, that is, $\supp
(p)$ is compact.
\item The marginal density of $X$ is bounded away from zero within its
support, that is, $\fmin\triangleq\inf_{x\in\supp(p)} p(x)> 0$.
\item The response variables are almost surely bounded, that is, $\llvert {Y_i}\rrvert  \le Y_{\max}$ almost surely for some
non-random value $\Ymax> 0$.
\end{enumerate}

Next, recall step 2 in the procedure (P1), and define
\[
\cX_n^L \triangleq \bigl\{X_i: \hat
f_n(X_i) \ge L; 1\le i \le n \bigr\}. %
\]
The subset $\cX_n^L$ can be used to construct an estimator to the
level set $D_L$:
%
%
%e3.3 #&#
\begin{equation}
\label{eq:set_est} \hat D_L(n,r) \triangleq U\bigl(\cX_n^L,
r\bigr).
\end{equation}
Note that the radius $r$
is the same $r$ as used for the bandwidth of the kernel function. This
connection is crucial for the proofs.

To overcome the noisiness of the estimator $\hat D_L(n,r)$ discussed
above, we present the following procedure.
First, note that
for any $\eps\in(0,L)$, we have that $\hat D_{L+\eps}(n,r) \subset
\hat D_{L-\eps}(n,r) $. The inclusion map
\[
\imath:\hat D_{L+\eps}(n,r) \hookrightarrow\hat D_{L-\eps}(n,r)
\]
induces a map in homology
%
%
%e3.4 #&#
\begin{equation}
\label{eq:i_star} \imath_*:H_*\bigl(\hat D_{L+\eps}(n,r)\bigr) \to H_*\bigl(
\hat D_{L-\eps}(n,r)\bigr).
\end{equation}
We use this map to define
%
%
%e3.5 #&#
\begin{equation}
\label{eq:def_est} \hat{H}_*(L,\eps; n) \triangleq\im(\imath_*).
\end{equation}
We will use $\hat{H}_*(L,\eps; n)$ as an estimator for $H_*(D_L)$.
The intuition behind using this inclusion map is as follows. Using
Lemma~\ref{lem:set_bounds}, we can show that with a high probability
we have
%
%
%e3.6 #&#
\begin{equation}
\label{eq:level_inclusion} %
\begin{array} {ccccccccc} D_{L+2\eps} & & & &
D_L & & & & D_{L-2\eps}
\\
& \hookurarrow& & \hookdrarrow& & \hookurarrow& & \hookdrarrow&
\\
& & \hat D_{L+\eps}(n,r) & & \stackrel{\imath} {\hookrightarrow} & & \hat
D_{L-\eps}(n,r) & & \end{array} %
,
\end{equation}
where $\hookrightarrow$ represents inclusion.
Assuming that $H_*(D_{L+2\eps}) \cong H_*(D_L) \cong H_*( D_{L-2\eps
})$, then all the cycles in $H_*(D_L)$ must\vspace*{1pt} persist throughout this
entire sequence of inclusions and in particular they should be present
in $\hat{H}_*(L,\eps; n)$. In contrast, any cycles in $\hat D_{L\pm
\eps}(n,r)$ that do not belong to $D_L$ must be terminated as we move
from $\hat D_{L+\eps}(n,r)$ to $\hat D_{L-\eps}(n,r)$ via $D_L$, and
therefore should not be in $\hat{H}_*(L,\eps; n)$. To prove that the
inclusion sequence in \eqref{eq:level_inclusion} holds, we require the
following regularity condition on $L$.

%%%%%%%%%%

%
%de3.2 #&#
\begin{defn}
Given a level $L>0$ and $\eps\in(0,L/2)$, we say that $L$ is $\eps
$-regular if
\begin{eqnarray*}
\partial D_{{L+2\eps}} \cap\partial D_{{L+\sklfrac{3}{2}\eps}} &=& \partial
D_{{L+\sklfrac{1}{2}\eps}}\cap\partial D_{L} = \partial D_{L} \cap
\partial D_{{L-\sklfrac{1}{2}\eps}}
\\
&=& \partial D_{{L-\sklfrac{3}{2}\eps}}\cap\partial
D_{{L-2\eps}} = \varnothing,
\end{eqnarray*}
where ``$\partial$'' is the set boundary.
\end{defn}

%%%%%%%%%%

This regularity condition basically guarantees sufficient ``separation''
between the level sets involved in the estimation process (its\vspace*{1pt}
importance will become clearer in the proofs).
In particular, if $f$ is continuous in $f^{-1}([L-2\eps,L+2\eps])$,
then $L$ is $\eps$-regular.
We will assume that the levels we are studying are always $\eps$-regular.

We now state the main result in this paper which holds for both the
density estimation as well as regression setting.

%th3.3 #&#
\begin{teo}\label{teo:persistence}
Let $L > 0$ and $\eps\in(0,L/2)$ be such that the function $f(x)$ has
no critical values in the range $[L-2\eps, L+2\eps]$.
If $r\to0$, and $nr^d\to\infty$, then for $n$ large enough we have
\[
\mathbb{P} \bigl(\hat H_*(L,\eps;n) \cong H_*(D_L) \bigr)
\ge1 - 6 ne^{-C^\star_{{\eps/2}} n r^d}.
\]
In particular, if $nr^d \ge D \log n$ with $D > (C^\star_{\eps
/2})^{-1}$, then
\[
\lim_{\ninf}\mathbb{P} \bigl(\hat H_*(L,\eps; n)
\cong H_*(D_L) \bigr) = 1.
\]
\end{teo}

The constant value $C^\star_\eps$ in the theorem above is
%
%
%e3.7 #&#
\begin{equation}
\label{eq:C_star_density} C^\star_\eps= \frac{\eps^2 C_K}{3{\fmax}+ \eps},
\end{equation}
for density estimation, and
%
%
%e3.8 #&#
\begin{equation}
\label{eq:C_star_regression} C^\star_\eps= \frac{\eps^2 \fmin^2 C_K}{3(\Ymax^2+\eps^2)\fmax
+2\eps\fmin(\Ymax+\eps)},
\end{equation}
for regression (see the \hyperref[append]{Appendix} for more details).

Theorem~\ref{teo:persistence} states that if we want to recover the
homology of the level set $D_L$ we can compute the image of the
homology map as we move from $\hat D_{L+\eps}(n,r)$ to the slightly
larger complex $\hat D_{L-\eps}(n,r)$.
We note that another possible solution to this estimation problem is to
dilate the estimated set $\hat D_L$ directly (e.g.,~by covering the
points with a slightly larger balls), as suggested by the results in
\cite{chazalsampling2009}. However, such a method will require
further knowledge about the level sets (such as their feature size),
and the gradient of the function $f$, which is not required by the
method we propose here.

\begin{rem*}
In order to choose $D$, we need to know the values of $\fmin,\fmax$
and $Y_{\max}$, which might not be directly available.
There are a few possible ways to address this problem:
\begin{enumerate}
\item
Since all we need are bounds and not the precise values, one option is
to make the broad assumption that $p$ belongs to a class of density
functions bounded by some fixed values, and use a similar assumption
for $Y$.
\item Another option is to estimate these values from the data, taking
values as high as we want for the upper bounds $\fmax$, $Y_{\max}$ (and
as low as we want for the lower bound $\fmin$), to guarantee that the
estimated values are indeed valid bounds with high probability. Using
estimated values instead of the true ones affects the theoretical
validity of Theorem~\ref{teo:persistence}, but we believe it should
have a negligible effect in practice.
\item Finally, another option is to take $nr^d \gg\log n$ (e.g., $nr^d
= (\log n)^2$). Then it is guaranteed that the probability converges to
one, and we do not need to know the value of $C_{\eps}$.
\end{enumerate}
\end{rem*}

%We shall describe an algorithm to compute this homology in Section
%\ref{sec:compute}.

In the following sections we describe two applications for the
estimator we proposed, addressing problems that are of significant
interest in the fields of topological data analysis and machine learning.

%%%%%%%%%%

%s3.1 #&#
\subsection{An application to manifold learning}

%%%%%%%%%%

Let $\cM$ be a smooth $m$-dimensional, closed manifold (compact and
without a boundary), embedded in $\R^d$.
Given a random sample $\cX_n =  \{X_1,\ldots, X_n \}
\subset\R^d$
we wish to recover the homology of $\cM$.
The case where the observations are drawn directly from the manifold
(i.e., $\cX_n \subset\cM$), has been extensively studied (see \cite
{bobrowskitopology2014,niyogifinding2008}). In \cite
{bobrowskitopology2014}, the following asymptotic result was presented.

%th3.4 #&#
\begin{teo}[(Theorem 4.9 in \cite{bobrowskitopology2014})]
If $nr^d \ge C \log n$, and $C > (\omega_d \fmin)^{-1}$, then:
\[
\lim_{\ninf}\mathbb{P} \bigl(H_*\bigl(U(
\cX_n, r)\bigr) \cong H_*(\cM) \bigr) = 1,
\]
where $\omega_d$ is the volume of a $d$-dimensional unit ball, and
$\fmin= \inf_{x\in\cM} p(x) > 0$.
\end{teo}

In this section, we extend this result to the case where noise is
present. The term ``noise'' in this context refers to the fact that the
observations do not necessarily lie on $\mathcal M$, but rather in its vicinity. As
an example, consider the observations $X_1,\ldots,X_n$ defined as
%
%
%e3.9 #&#
\begin{equation}
\label{eq:ex_dist} X_i = Y_i + Z_i\qquad\mbox{where } Y_i \stackrel{\iid} {\sim} \rho (\cM)\quad\mbox{and}\quad
Z_i \stackrel{\iid} {\sim} \mbox{N}\bigl(\mathbf {0},
\sigma^2 \mathbf{I}_d\bigr),
\end{equation}
where\vspace*{1pt} $Y_i$ is drawn from a distribution $\rho$ that is supported on a
manifold $\cM$, and $Z_i$ is drawn from the normal distribution in the
ambient space $\R^d$. For this model, the methods used to prove
consistency of the estimator in
\cite{bobrowskitopology2014,niyogifinding2008} no longer apply
since the outliers produced by the noise create their own topology, and
interfere with our ability to recover $H_*(\cM)$.

The seminal work in \cite{niyogitopological2011} studies the
following special case. Let $Y_i \stackrel{\iid}{\sim} \rho(\cM)$,
for each
$i \in \{1,\ldots,n \}$ let $\cN_i$ be the normal space
to $\cM$
at $Y_i$, and let $Z_i \sim\mbox{N}(0,\sigma^2 \bI_{d-m})$ be a
multivariate normal variable in the normal space $\cN_i$. Our
observations are then taken to be $X_i = Y_i + Z_i$. Under explicit
assumptions on $\sigma$ and $\cM$, they show that the homology of
$\cM$ can be recovered from $\cX_n$ with a high probability. The work
in \cite{balakrishnanminimax2012} extends this idea to a few other
noise models. The results and proofs in \cite
{balakrishnanminimax2012,niyogitopological2011} are tied to
specific noise models and rely on the parameters of the noise model and
the geometry of $\cM$. We wish to use the result in Theorem~\ref
{teo:persistence} to study the same homology inference problem for a
large class of distributions, and with as few assumptions as possible.

We start by defining a general class of density functions on $\R^d$,
from which it would be possible to extract the homology of $\cM$.

%%%%%%%%%%

%
%de3.5 #&#
\begin{defn}\label{def:noisy_dens}
Let $p:\R^d\to\R_+$ be a probability density function. We say that
$p$ \emph{represents a noisy version} of $\cM$, if there exist $0 < A
< B < \infty$ such that:
\begin{enumerate}
\item For every $L\in[A,B]$ we have $D_L \simeq\cM$.
\item For every $L > B$, we have $D_L \simeq\cM'$, where $\cM'
\subset\cM$ is a compact locally contractible proper subset of $\cM$,
\end{enumerate}
where ``$\simeq$'' stands for homotopy equivalence (see Section~\ref
{sec:topdefs}).
\end{defn}

%%%%%%%%%%

In other words, we consider density functions $p$ for which there is a
range where the level sets are ``similar'' to $\cM$. For levels higher
than this range, the level sets are ``similar'' to nice subsets of $\cM
$. For example, the distribution in \eqref{eq:ex_dist} satisfies this
conditions for small enough~$\sigma$.
By ``locally contractible'' we refer to the property that every point $x$
has a neighborhood $\cN_x$ that is homotopy equivalent to a single
point. For example, if $\cM'$ is a compact manifold with boundary,
then it is locally contractible. We need this requirement to rule out
the appearance of highly twisted topological spaces.
In Figure~\ref{fig:noisy_ver}, we present a sequence of level sets for
a density function that represents a noisy version of the torus. This
density was generated by taking a uniform distribution on the latitude
angle, a wrapped normal distribution on the longitude angle, and adding
independent Gaussian noise. Note that the level sets are
$3$-dimensional whereas the torus is $2$-dimensional. Nevertheless, we
can see that there is a whole range of levels where they are
topologically equivalent.

%%%%%%%%%%

%
%f6 #&#
\begin{figure}

\includegraphics{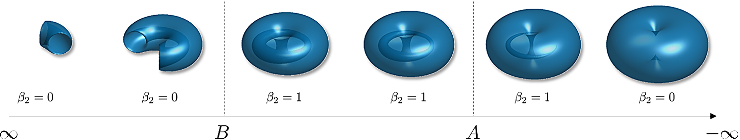}

\caption{In this figure we demonstrate a sequence of level sets for a
density function $p$ that is a noisy version of the $2$-dimensional
torus. The horizontal axis represents the function levels in a
decreasing order. For very high values ($L > B$) we see that the level
sets look like a subset of the torus. Note that they are not real
subsets, since these are $3$-dimensional shapes, whereas the torus is
$2$-dimensional. Inside the range $(A,B)$ the level sets look like the
torus (where $\beta_0 = \beta_2 = 1$, and $\beta_1 = 2$). For low
levels the topology changes again, be we no longer require any assumptions.}
\label{fig:noisy_ver}
\end{figure}

%%%%%%%%%%

The model described in Definition~\ref{def:noisy_dens}
generalizes the additive Gaussian noise model discussed in \cite
{balakrishnanminimax2012,niyogitopological2011} but is essentially
different than the other noise models in \cite
{balakrishnanminimax2012}. This model is very broad in the sense that
it is not tied to any specific assumptions on the distribution (e.g.,
uniform in the ``clutter'' and ``tubular'' noise models, or having
Fourier transform bounded away from zero in the ``additive'' model \cite
{balakrishnanminimax2012}). In addition, we believe that this model
is more ``natural'' for topological estimation since it emphasizes the
topological behavior of the density rather than making analytic
assumptions on its functional structure.

If we know a priori the values of $A$ and $B$, then the recovery method
would be simple.
Given a sample $\cX_n =  \{X_1,\ldots, X_n \}\stackrel
{\iid}{\sim}
p$, and setting $f=p$, we choose $L$ and $\eps$ such that $[L-2\eps,
L+2\eps] \subset(A,B)$, and compute $ \hat H_*(L,\eps; n)$.
Theorem~\ref{teo:persistence} guarantees that with high probability
$\hat H_*(L,\eps; n)\cong H_*(D_L)\cong H_*(\cM)$.

However, in real problems we are not given $A,B$ so the real challenge
is to recover $\cM$ without knowing the stable range.
To show that the procedure described below is consistent, we require
the following assumptions to hold.
\begin{enumerate}[(ii)]
\item[(i)] $\cM$ is connected and orientable;
\item[(ii)] $B-A > 8\eps$.
%\item[(iii)] There exists a value $L_{\max}$, such that $f\le L_{
%\max}$.
\end{enumerate}

The following procedure (P2) will be used to estimate the homology of
$\cM$ from the a noisy sample $\cX_n$.
In this procedure, we will use the estimated Betti numbers defined as
$\hat\beta_k(L,\eps;n) \triangleq\rank(\hat H_*(L,\eps;n))$.
Define
%
%
%e3.10 #&#
\begin{equation}
\label{eq:levels} N_\eps:= \sup_{x\in\R^d} \bigl
\lceil{f(x)/2\eps} \bigr\rceil,\qquad L_{\max} = 2\eps N_\eps\quad\mbox{and}\quad L_i = L_{\max}-2i\eps.
\end{equation}
The procedure (P2) is as follows.
% The following procedure (P2) gives us an estimate of the homology of
%the manifold from the noisy samples $\cX_n$.
%
\begin{enumerate}
% Set a fine enough grid on the level set values
%$$\Lmax= L_0 > L_2 > \cdots> L_{N_\eps} = 0,$$
%such that $L_i - L_{i+1} = 2\eps$ for all $1 \le i \le M-1$.
%
\item Compute $\hat H_*(L_i,\eps;n)$ for all $i=1,\ldots,N_\eps$.
\item Define
\[
i^\star\triangleq1+\min \bigl\{i\in \{1,\ldots,N_\eps \}:\hat
\beta_m(L_i,\eps;n) = 1 \bigr\}.
\]
This index will be shown to be the first point where we are guaranteed
to observe the homology of $\cM$.
\item Our estimator for the homology of $\cM$ will then be $\hat
H_*(L_{i^\star}, \eps;n)$.
\end{enumerate}
Note that in this procedure a choice has to be made for the parameter
$r$ (the radius of the balls and the bandwidth of the kernel). The
following theorem states that if $r$ is chosen appropriately we can
estimate the homology of a manifold from noisy observations with high
probability.

%%%%%%%%%%

%
%th3.6 #&#
\begin{teo}\label{teo:manires}
Let $\cM$ be a $m$-dimensional closed, connected, orientable manifold
embedded in~$\R^d$.
Let $X_1,\ldots, X_n$ be data points sampled from a density function
$p$ satisfying the conditions in Definition~\ref{def:noisy_dens}.
Choose $r\to0$ that satisfies $nr^d \ge D \log n$ with $D > (C^\star
_{\eps/2})^{-1}$, where $C_\eps$ is defined in \eqref
{eq:C_star_density}. Applying procedure \textup{(P2)}, we then have
\[
\lim_{\ninf}\mathbb{P} \bigl(\hat
H_*(L_{i^\star},\eps;n) \cong H_*(\cM) \bigr) = 1.
\]
\end{teo}

%%%%%%%%%%

We state here the main ideas used in proving the above, while the
detailed proof is given in the \hyperref[append]{Appendix}. We use Poincar\'{e} duality,
a fundamental idea in algebraic topology. Poincar\'{e} duality relates
homology groups to co-homology groups of closed orientable
$m$-dimensional manifolds, stating that
$H_k(\cM) \cong H^{m-k}(\cM)$, where $H^{m-k}(\cM)$ is the
co-homology of $\cM$ (cf. \cite
{hatcheralgebraic2002,munkreselements1984}).
An important consequence of Poincar\'{e} duality is that $\beta_k (\cM
) = \beta_{m-k}(\cM)$ for every $k=0,\ldots,m$, and in particular
$\beta_0(\cM) = \beta_m(\cM)$.
Our assumption that $\cM$ is connected implies that $\beta_0(\cM) =
1$, and from Poincar\'{e} duality we conclude that $\beta_m(\cM)=1$
as well. In contrast, if $\cM'\subset\cM$ is a proper compact
locally contractible subset of $\cM$ then using a different type of
duality one can show that $\beta_m(\cM') = 0$ (see Proposition 3.46
in \cite{hatcheralgebraic2002}).
Our assumptions on $A,B$ then implies that if $L_i >B$ we have $\beta
_m(D_{L_i}) = 0$, while if $L_i\in(A,B)$ then $\beta_m(D_{L_i}) = 1$.
Therefore, the first $L_i$ for which the $m$-th Betti number switches
from $0$ to $1$ necessarily lies in $(A,B)$, and we can use this $L_i$
to recover the homology of $\cM$. In practice, we defined $i^\star$
to be the second level at which we have $\hat\beta_m(L_i,\eps;n)=1$.
This is a precautionary measure which we discuss in the proof.

%%%%%%%%%%

\begin{rem*}
\begin{enumerate}[3.]
\item To use the result in Theorem~\ref{teo:manires} one needs to know
the values of $m$ and $\eps$. We consider these values to be crucial
information required to ``extract'' the topology of the manifold. Their
knowledge replaces other assumptions about the geometry of the manifold
which we want to avoid. Note that for $\eps$ we do not require a
precise value but any lower bound would suffice.
\item Also required is the knowledge $L_{\max}$ (or equivalently
$N_\eps$). Note, that when we have a finite sample $ \{X_1,\ldots, X_n \}$ we can estimate $L_{\max}$ using $\hat L_{\max}:=
\max_{i}
 \lceil{f_n(X_i)/2\eps} \rceil$. For every $L> \hat
L_{\max}$ we have $\hat
D_{L}(n,r) = \varnothing$. Therefore, in practice, even if the true
$L_{\max}$ is higher than $\hat L_{\max}$, it does not affect the
procedure, since the higher levels are empty anyway.
\item It is possible that small perturbations in the density function
will generate $m$-dimensional cycles at level sets with $L>B$. To be
able to ignore these cycles when they appear, additional information
about the geometry of the underlying manifold should be provided
(e.g.,~its feature size), otherwise it will be impossible to determine
which of the \mbox{$m$-}dimensional cycles belongs to the manifold (even if
the function $f$ is known completely), and the homology inference
problem is ill-posed. If we want to limit ourselves to use only the
fact that the data is ``concentrated'' around a $m$-dimensional
manifold, then we need to assume the density function allows us to
identify it properly, and that is the essence of Definition~\ref
{def:noisy_dens}.
\end{enumerate}
\end{rem*}

%%%%%%%%%%

%s3.2 #&#
\subsection{Persistent homology and application to clustering}\label
{sec:persistent_homology}

%%%%%%%%%%

A common topological summary used in TDA is persistent homology (see
Section~\ref{sec:topdefs}). Given a function $f$ the persistent
homology of $f$, $\PH_*(f)$,
tracks when the homology of (super) level-sets of $f$ changes and
serves as a summary of the function. This summary contains information
about the creation and destruction of connected components and cycles
of the level sets. In the case where $f=p$ is a density function, the
zeroth persistent homology $\PH_0(f)$ can viewed a summary of the
evolution of clusters in the data, and can be useful for clustering
algorithms as discussed in Section~\ref{sec:intro_ph}.
By definition, $\PH_*(f)$ is computed from the continuous
filtration $\cD=  \{D_L \}_{L\in\R}$ as $L$ decreases
from $\infty
$ to $-\infty$.
Note that the persistent homology $\PH_*(f)$ contains much more
information than just the homology at each level $D_L$. It also
contains information about mappings between different levels, and hence
enables us to track the evolution of cycles.

In this section, we wish to address the estimation of $\PH_*(f)$ where
$f:\R^d\to\R$ is either a density function (tame and bounded) or a
regression function (satisfying the conditions (C2) as well).
In both cases, we have shown that the estimator $\hat H_*(L, \eps, n)$
defined in \eqref{eq:def_est}, can recover the homology of $D_{L}$ for
every $L$. In order to recover the persistent homology we also need to
make sure that the mappings between different levels are recovered as
well. The error measure we use is the commonly used ``bottleneck
distance'' (see Section~\ref{sec:topdefs}).
To estimate $\PH_*(f)$, recall the definitions of $N_\eps,L_{\max}$,
and $L_i$ in \eqref{eq:levels} and consider the following discrete filtration
\[
\hat{\cD}^\eps\triangleq \bigl\{\hat D_{L_i}(n,r) \bigr
\}_{i\in\Z},
\]
where $\hat D_{L_i}(n,r)$ is defined by \eqref{eq:set_est}.
Denoting the persistent homology of $\hat{\cD}^\eps$ by $\hPH
_*^\eps(f)$, and using the methods presented in this section we prove
the following.

%%%%%%%%%%

%
%th3.7 #&#
\begin{teo}\label{teo:ph_consistency}
If $r\to0$ and $nr^d\to\infty$, then
\[
\mathbb{P} \bigl(d_B \bigl(\hPH_*^\eps(f), {
\PH}_*(f) \bigr) \le 5\eps \bigr) \ge1-3 N_\eps ne^{-C^\star_{\eps/2}nr^d},
\]
where $C^\star_\eps$ is defined in \eqref{eq:C_star_density}
(density) and \eqref{eq:C_star_regression} (regression).
In particular, if $nr^d \ge D\log n$ with $D > (C^{\star}_{\eps
/2})^{-1}$, we have
\[
\lim_{\ninf}\mathbb{P} \bigl(d_B \bigl(
\hPH_*^\eps(f), {\PH}_*(f) \bigr) \le 5\eps \bigr) =1.
\]
\end{teo}

%%%%%%%%%%

In other words, we state that the estimator $\hPH_*^\eps(f)$ is
``consistent'' up to a given precision of $5\eps$.
Note that we will always have some discretization error since our
estimator is discrete (having an inherent step size $\eps$) while the
filtration we wish to study is continuous. However, one can make $\eps
$ arbitrarily small to achieve higher precision. The smaller value of
$\eps$ we choose the smaller $C^\star_{\eps/2}$ will be and the
convergence of $\hPH_*^\eps(f)$ to $ {\PH}_*(f)$ will be slower.

To prove this theorem (see \hyperref[append]{Appendix}), we invoke Lemma~\ref
{lem:set_bounds} $M$ times in order to form a sequence of inclusions
alternating between level sets $D_L$ and their estimates $\hat
D_L(n,r)$. This alternating sequence is called ``interleaving'' and the
work in \cite{chazalproximity2009} provides means to bound the
distance between the persistent homology computed for these two types
of filtrations. In Section~\ref{simulations}, we provide several
examples for the estimation of persistent homology using $\hPH_*^\eps(f)$.

As we discuss in Section~\ref{sec:compute}, Theorem~\ref
{teo:ph_consistency} can be adjusted to use the filtration of Rips
complexes $ \{R_{L_i}(n,r) \}_{i \in\Z}$ instead of $\{
\hat
D_{L_i}(n,r)\}_{i \in\Z}$.
The work in \cite{chazalscalar2011,chazalpersistence-based2013}
studies a different method to recover the persistent homology of $f$
using Rips complexes. In order to recover $\PH_*(f)$, \cite
{chazalscalar2011} considers the maps $\imath
_*^L:H_*(R_L(n,r))\hookrightarrow H_*(R_L(n, 2r))$ induced by inclusion
for all values of $L$ and for a fixed $r$. The persistence module for
the family of images -- $ \{\im(\imath_*^L) \}_L$ is then
used as an
approximation for $\PH_*(f)$.
In a way, one can think of the transition $R_L(n,r) \hookrightarrow
R_L(n, 2r)$ as playing the same role as the transition $R_{L+\eps
}(n,r) \hookrightarrow R_{L-\eps}(n,r)$ we study in this paper,
``filtering'' the noisy homology.
Changing the radius rather than the level, allows one to avoid the
level discretization that our method relies on, which leads to a more
accurate approximation. On the other hand, this method requires further
assumptions on the model parameters, and computing the estimator is
more complicated. It remains future work to study whether these two
methods could be combined into a more powerful and robust one.

In a different line of work \cite
{bubenikstatistical2010,chungpersistence2009,fasyconfidence2014}
persistent homology is recovered by constructing a kernel-based
estimator $\hat f$ for the function at hand and then computing the
persistent homology of the estimator $\PH(\hat f)$. The work in \cite
{phillipsgeometric2013} presents a different approach by recovering
the sublevel sets of distance-like functions called ``kernel distance'' functions.
The validity of these methods is established by using the stability
theorem \cite{cohen-steinerstability2007} stating that $d_B(\PH
_*(f), \PH_*(\hat f)) \le\Vert{f-\hat f}\Vert_\infty$. There are two
significant advantages to the estimator we propose in this paper.
First, we do not require assumptions about the global sup-norm
convergence of the estimator. Second, computing the estimator $\PH
(\hat f)$ in practice involves discretizing the space, and this may
have a significant effect on the ability to recover small features in
the data (see, e.g., the clustering examples in Section~\ref
{simulations}). The estimator we propose does not require such a discretization.

%%%%%%%%%%

%s4 #&#
\section{Computing the homology estimator}\label{sec:compute}

%%%%%%%%%%

The estimator we propose in Section~\ref{sec:results} requires the
computation of the image between the homology groups of $\hat D_{L+\eps
}(n,r)$ and $\hat D_{L-\eps}(n,r)$ (defined in \eqref{eq:set_est}).
As a review for a more statistical audience, we state the fundamental
tools required to compute this estimator. In general, algorithms for
computing homology of unions of balls require two steps. The first step
is to obtain a combinatorial representation of the geometric object
that is either equivalent in homology or approximately equivalent in
homology to the original geometric object. This step is outlined in
Section~\ref{complexes}.
The combinatorial representation reduces homology computation to a
linear algebra problem. The second step is to apply a set of linear
transformations to this combinatorial representation to compute the
image of the homology groups under the inclusion map between two
complexes. This step is outlined in
Section~\ref{sec:alg}.

%%%%%%%%%%

%s4.1 #&#
\subsection{The \v{C}ech and Vietoris--Rips complex}\label{complexes}

%%%%%%%%%%

Let $S$ be a set, and $\Sigma\subset2^S$ be a collection of finite
subsets of $S$.
We say that $\Sigma$ is an \emph{abstract simplicial complex} if for
every $A\in\Sigma$ and $B\subset A$ we also have $B\in\Sigma$. In
this section we introduce two special types of abstract simplicial
complex that can be useful for computing the estimators presented in
this paper.

Let $\cX= \{x_1,\ldots,x_n \}$ be a set of points in $\R
^d$, and
suppose that we wish to compute the homology of the union of balls
$U(\cX,r)$ (see \eqref{eq:union_balls}) for some $r>0$.
The \v{C}ech complex is an abstract simplicial complex that allows us to
convert the homology computation problem into linear algebra.
The Vietoris--Rips (or just Rips) complex can be thought of as an
approximation to the \v{C}ech complex. This approximation offers
computational advantages over the \v{C}ech complex but suffers from not
sharing the same direct relation to the homology of $U(\cX,r)$ as the
\v{C}ech complex.
We first provide the definitions for these complexes.

%de4.1 #&#
\begin{defn}[(\v{C}ech complex)]\label{def:cech_complex}
Let $\cX=  \{x_1,x_2,\ldots,x_n \}$ be a collection of
points in $\R
^d$, and let $r>0$. The \v{C}ech complex $C(\cX, r)$ is constructed
as follows:
\begin{enumerate}
\item The $0$-simplices (vertices) are the points in $\cX$.
\item A $k$-simplex $[x_{i_0},\ldots,x_{i_k}]$ is in $C(\cX,r)$ if
$\bigcap_{j=0}^{k} {B_{r}(x_{i_j})} \ne\varnothing$.
\end{enumerate}
\end{defn}

%%%%%%%%%%

%
%de4.2 #&#
\begin{defn}[(Rips complex)]\label{def:rips_complex}
Let $\cX=  \{x_1,x_2,\ldots,x_n \}$ be a collection of
points in $\R
^d$, and let $r>0$. The Rips complex $R(\cX, r)$ is constructed as follows:
\begin{enumerate}
\item The $0$-simplices (vertices) are the points in $\cX$.
\item A $k$-simplex $[x_{i_0},\ldots,x_{i_k}]$ is in $R(\cX,r)$ if
$\llVert  x_{i_j} - x_{i_l}\rrVert  \le2r$ for all $0\le j,l
\le k$.
\end{enumerate}
\end{defn}

%%%%%%%%%%

Figure~\ref{fig:cech} depicts a simple example of a \v{C}ech and Rips
complex in $\R^2$. The figure also highlights the contrast between the
two complexes. The main difference is that the Rips complex is
constructed simply from pairwise intersection information while the
\v{C}ech complex requires high-order information. This difference is
realized in Figure~\ref{fig:cech} in the far left triangle in either
complex. In the Rips complex, the left triangle is filled in to be a
face, since all three pairwise intersections occur. In the \v{C}ech
complex higher-order interactions
are also computed, in this case one observes that the three pairwise
intersections do not overlap resulting in three edges rather than a
filled in face. The main advantage of the Rips complex is computational~-- all we need in order to construct the Rips complex is to compute the
pairwise distances between all the points, rather than to check for all
possible orders of intersections of balls as we would have to for the
\v{C}ech complex.

The Rips complex can be considered as an approximation to the \v{C}ech
complex. It is clear from the definitions that $C(\cX,r) \subset R(\cX
,r)$. In addition, it is shown in \cite{desilvacoverage2007} that
$R(\cX,r) \subset C(\cX, \sqrt{2} r)$. Combining these two
statements we have that
\[
R(\cX,r) \subset C(\cX,\sqrt{2}r) \subset R(\cX,\sqrt{2} r).
\]

%%%%%%%%%%

%
%f7 #&#
\begin{figure}

\includegraphics{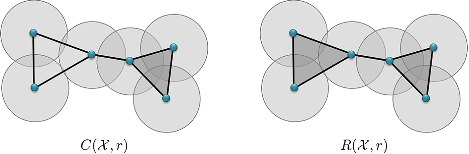}

\caption{On the left -- the \v{C}ech complex $C(\cX,r)$, on the right -- the Rips complex $R(\cX,r)$ with the same set of vertices and the same
radius. We see that the three left-most balls do not have a common
intersection and therefore do not generate a 2-dimensional face in the
\v{C}ech complex. However, since all the pairwise intersections occur, the
Rips complex does include the corresponding face.}
\label{fig:cech}
\end{figure}

%%%%%%%%%%

An important result in algebraic topology called the ``Nerve lemma'' (cf.
\cite{borsukimbedding1948}) states that the \v{C}ech complex $C(\cX
,r)$ is homotopy equivalent to the neighborhood set $U(\cX,r)$. In
particular it follows $H_*(C(\cX,r)) \cong H_*(U(\cX,r))$. As a
consequence, any statement made about the homology of $U(\cX,r)$
applies to $C(\cX,r)$ and vice versa.

Denote the \v{C}ech complex generated by the filtered point set $\cX_n^L$
as $C_L(n,r) \triangleq C(\cX_n^L,r)$. We can then define
\[
\imath_*: H_*\bigl(C_{L+\eps}(n,r)\bigr)\to H_*\bigl(C_{L-\eps}(n,r)
\bigr)
\]
to be the map induced by the inclusion map between the simplicial complexes.
Defining
\[
\hat H_*^C(L,\eps;n) \triangleq\im(\imath_*),
\]
then by the Nerve lemma, since $\hat D_{L\pm\eps}(n,r)$ and $C_{L\pm
\eps}(n, r)$ are completely equivalent structures (in terms of homology),
Theorem~\ref{teo:persistence} holds without changes for $\hat
H_*^C(L,\eps;n)$.

Next, we denote the Rips complex constructed from the filtered sample
as $R_L(n,r) \triangleq R(\cX_n^L, r)$ and define the following
inclusion map for any $\eps\in(0,L/2)$
\[
\imath:R_{L+\eps}(n,r) \hookrightarrow R_{L-\eps}(n,r). %
\]
This inclusion induces a map in homology
\[
\imath_*:H_*\bigl(R_{L+\eps}(n,r)\bigr)\to H_*\bigl(R_{L-\eps}(n,r)
\bigr),
\]
and we denote
\[
\hat H_*^R(L, \eps; n) \triangleq\im(\imath_*).
\]
Note that the Nerve lemma applies only to the \v{C}ech complex and not
the Rips.
Nevertheless, the following theorem states that we can compute the
homology of $D_L$ using the Rips complex as well. The importance of
providing a consistent estimator for $H_*(D_L)$ that uses the Rips
complex is due to its computational efficiency.

%%%%%%%%%%

%
%th4.3 #&#
\begin{teo}\label{teo:persistence_rips}
Let $L > 0$ and $\eps\in(0,L/2)$ be such that the function $f(x)$ has
no critical values in the range $[L-2\eps, L+2\eps]$.
If $r\to0$ and $nr^d\to\infty$, then for $n$ large enough we have
\[
\mathbb{P} \bigl(\hat H_*^R(L,\eps; n) \cong
H_*(D_L) \bigr) \ge1 - 6ne^{-C^\star
_{\eps/2} nr^d}.
\]
In particular, if $nr^d \ge D \log n$ with $D > (C^\star_{\eps
/2})^{-1}$, then
\[
\lim_{\ninf}\mathbb{P} \bigl(\hat
H_*^R(L, \eps; n) \cong H_*(D_L) \bigr) = 1.
\]
\end{teo}

%%%%%%%%%%

In the next subsection, we provide an algorithm for computing the image
of the inclusion map using either the \v{C}ech or Rips complex.

%%%%%%%%%%

%s4.2 #&#
\subsection{Computing the homology of the image}
\label{sec:alg}

%%%%%%%%%%

Our\vspace*{1pt} estimator for $H_k(D_L)$ requires the computation of the image of
the map between the homology of two nested simplicial complexes $\Delta
^{(1)}\subset\Delta^{(2)}$ (either \v{C}ech or Rips). This map is
denoted by $\imath_k:H_k(\Delta^{(1)}) \to H_k(\Delta^{(2)})$. In
this section, we present an algebraic algorithm to compute the rank of
this image, namely the estimated Betti number $\beta_k$. Note that
there are several efficient algorithms to compute persistent homology
that can also be used here (see \cite
{adamsjavaplex2014,edelsbrunnercomputational2010,mischaikowmorse2013}).
We present a relatively simple algorithm, in the interest of clarity for
a statistical audience, for the case where $\F$ is a field of
characteristic zero (e.g., $\R, \Q$). For a fixed homology degree
$0\le k \le d$ the algorithm will consist of two steps:
\begin{enumerate}[(2)]
\item[(1)] Finding a basis for the kernel of a square matrix defined
later as $L_k^{(1)}$.
\item[(2)] Computing the rank of two matrices, defined later as
$\partial_{k+1}^{(2)}$ and $\hat\partial_{k+1}^{(2)}$, and then we
will have that
\[
\rank\bigl(\im(\imath_k)\bigr) = \rank\bigl(\hat
\partial_{k+1}^{(2)}\bigr) - \rank \bigl(\partial_{k+1}^{(2)}
\bigr).
\]
\end{enumerate}
In the\vspace*{1pt} following, we provide more details about homology computation
for simplicial complexes, and in particular the definitions of the
matrices $L_k^{(1)}$, $\partial_{k+1}^{(2)}$, and $\hat\partial
_{k+1}^{(2)}$ mentioned above.

%%%%%%%%%%

%s4.2.1 #&#
\subsubsection{Computing the homology of a simplicial complex}

%%%%%%%%%%

Let $\Delta$ be a simplicial complex, let $\Delta_k$ be the set of
$k$-simplexes in $\Delta$, and let $n_k = \llvert {\Delta
_k}\rrvert $, so we can write
\[
\Delta_k = \{\sigma_1, \sigma_2,\ldots,
\sigma_{n_k} \}.
\]
We assume that every $k$-simplex $\sigma_i \in\Delta_k$ is attached
with a unique orientation (an ordering on its set of vertices), denoted
by $\sigma_i = [x_0^i,\ldots, x_k^i]$. Defining $C_k \triangleq\F
^{n_k}$, we wish to map the simplexes of $\Delta_k$ into a basis of
$C_k$ in a way that preserves orientation information. To do that we
first define $\Delta_k^\pi$ to be the set containing all the
simplexes in $\Delta_k$ in all possible orientations. We then define
the map $T_k:\Delta_k^\pi\to C_k$ in the following way. For every
simplex $\sigma_i\in\Delta_k$ we define $T_k(\sigma_i) = \be_i$,
where $\be_i$ consists of one at the $i$th entry, and zero elsewhere.
For every permutation $\pi$ on $0,\ldots,k$, we then define
\[
T_k\bigl(\bigl[x^i_{\pi(0)}, \ldots,
x^i_{\pi(k)}\bigr]\bigr) = \sign(\pi) \be_i,
\]
where $\sign(\pi) = (-1)^{P(\pi)}$, and $P(\pi)$ is the parity of
the permutation $\pi$. The vector space $C_k$ is usually referred to
as the ``space of $k$-chains'' of $\Delta$.

Next, using the map $T_k$, we define the matrix $\partial_k$ to be a
$n_{k-1}\times n_k$ matrix where the $i$th column is given by
\[
(\partial_k)_i = \mathop{\sum
_{\sigma\in\Delta_{k-1}}}_{\mathrm
{is~a~face~of~}\sigma_i} T_{k-1}(\sigma).
\]
We note that the orientation of $\sigma$ used in the sum is the one
inherited from the orientation of $\sigma_i$.
In other words, the nonzero entries in the $i$th column correspond to
the $(k-1)$-dimensional faces of $\sigma_i\in\Delta_k$ (with the
proper sign representing their orientation). The matrix $\partial_k$
can be thought of as a linear transformation from $C_k$ to $C_{k-1}$
and is referred to as ``the boundary operator.'' The $k$th homology of
$\Delta$ is then defined to be the quotient space given by
%
%
%e4.1 #&#
\begin{equation}
\label{eq:homology} H_k(\Delta) \triangleq\ker(\partial_k) /
\im(\partial_{k+1}).
\end{equation}
One way to find a basis for $H_k(\Delta)$ is via the combinatorial
Laplacian, defined as the following $n_k\times n_k$ matrix
\[
L_k \triangleq\partial_{k+1}\partial_{k+1}^T
+ \partial_k^T\partial_k.
\]
Note that $L_0$ is the well-known graph Laplacian.
If $\F$ is a field with characteristic zero (e.g. $\R, \Q$) then it
is shown in \cite{friedmancomputing1998} that the kernel of $L_k$ is
isomorphic to $H_k(\Delta)$ and in particular, the Betti numbers of
$\Delta$ are given by $\beta_k(\Delta) = \dim(\ker(L_k))$.

%%%%%%%%%%

%s4.2.2 #&#
\subsubsection{The homology of the map}

%%%%%%%%%%

Our goal is not only to compute the homology of $\Delta^{(1)}$ and
$\Delta^{(2)}$ separately, but rather to compute the image of the map
$\imath_k:H_k(\Delta^{(1)})\to H_k(\Delta^{(2)})$.
For $j=1,2$ let $\Delta_k^{(j)}$ be the set of $k$-simplexes in
$\Delta^{(j)}$, and let $n_k^{(j)} = |\Delta_k^{(j)}| $.
Since $\Delta^{(1)} \subset\Delta^{(2)}$ we can list the simplexes
in the following way:
\begin{eqnarray*}
\Delta_k^{(1)} &=& \{\sigma_1,
\sigma_2, \ldots, \sigma _{n_k^{(1)}} \},
\\
\Delta_k^{(2)} &=& \{\sigma_1,
\sigma_2, \ldots, \sigma _{n_k^{(1)}}, \sigma_{n_k^{(1)}+1},
\ldots,\sigma_{n_k^{(2)}} \}.
\end{eqnarray*}
Using this ordering on the simplexes, we define the boundary operators
$\partial_k^{(j)}$ and the combinatorial Laplacians $L_k^{(j)}$ for
each of the complexes. It is then easy to see that
%
%
%e4.2 #&#
\begin{equation}
\label{eq:boundary_op} \partial_k^{(2)} = \lleft( %
\begin{array} {cc} \partial_k^{(1)} & \cdots
\\
0 & \ddots \end{array} %
 \rright).
\end{equation}
Now, if $ \{v_1, \ldots, v_m \} \subset C_k^{(1)}$ is a
basis for
$\ker(L_k^{(1)})$ then it represents a basis for $H_k(\Delta^{(1)})$,
such that $\beta_k(\Delta^{(1)}) = m$. Let $\hat v_i \in C_k^{(2)}$
be a zero padded version of $v_i \in C_k^{(1)}$. From \eqref
{eq:homology}, we know that $v_i \in\ker(\partial_k^{(1)})$, and
thus from \eqref{eq:boundary_op} it is clear that $\hat v_i \in\ker
(\partial_k^{(2)})$ as well. This implies that the vectors in $ \{
\hat v_1, \ldots, \hat v_m \}$ are candidates to form a basis
for $\im
(\imath_k)$. Note, however, that while $\hat v_i \in\ker(\partial
_k^{(2)})$, it is possible that some linear combinations of $\hat v_1,
\ldots, \hat v_m$ are in $\im(\partial_{k+1}^{(2)})$, which means
that they are considered as trivial in $H_k(\Delta^{(2)})$. This means
that $ \{\hat v_1,\ldots,\hat v_m \}$ might be larger than a
basis for
$\im(\imath_k)$, and we need to reduce this set. This can be done by
solving several sets of linear equations, which we avoid describing
here. However, the rank of $\im(\imath_k)$ can be computed easily by
\[
\rank\bigl(\im(\imath_k)\bigr) = \rank\bigl(\hat
\partial_{k+1}^{(2)}\bigr) - \rank \bigl(\partial_{k+1}^{(2)}
\bigr),
\]
where
\[
\hat\partial_{k+1}^{(2)} = \bigl(\partial_{k+1}^{(2)}
, \hat v_1, \ldots, \hat v_m\bigr)
\]
is a $n_k^{(2)}\times{(n_{k+1}^{(2)} + m)}$ matrix we get by
concatenating the boundary matrix $\partial_{k+1}^{(2)}$ with the
column vectors $\hat v_i$. In other words, we measure how many vectors
from the set $ \{\hat v_1,\ldots, \hat v_m \}$ can be added
to the
set of columns vectors of $\partial_{k+1}^{(2)}$ without generating
linear dependency.

%%%%%%%%%%

%s5 #&#
\section{Results on simulated data}\label{simulations}

%%%%%%%%%%

In this section, we illustrate how we can use the methods in
Section~\ref{sec:results} for data analysis using some simulated examples.
The examples we chose relate to classical problems in statistics:
classification, non-parametric regression, and clustering.
We use these examples to demonstrate the novelty and strength of the
methods proposed in this paper.

%%%%%%%%%%

%s5.1 #&#
\subsection{Binary regression}

%%%%%%%%%%

We illustrate how we can recover the homology of a classification
function. The marginal density of the explanatory variables
is uniform in the unit square $X \sim U ([-\frac{1}{2}, \frac
{1}{2}]^2 )$. We then set the conditional probability of the
binary response $Y$ as
%
%
%e5.1 #&#
\begin{equation}
\label{eq:ex_mu} \mathbb{P} (Y = 1\given X=x ) = f(x) \triangleq C\bigl(1+\sin
\bigl(4\pi\llVert x\rrVert ^2\bigr)\bigr) e^{-100(\llVert  x\rrVert -1/4)^2},
\end{equation}
where $C$ is a normalization factor guaranteeing that $f(x)$ is indeed
a conditional probability. The graph of this conditional probability
is given in Figure~\ref{fig:mu}.

%%%%%%%%%%

%
%f8 #&#
\begin{figure}[t]

\includegraphics{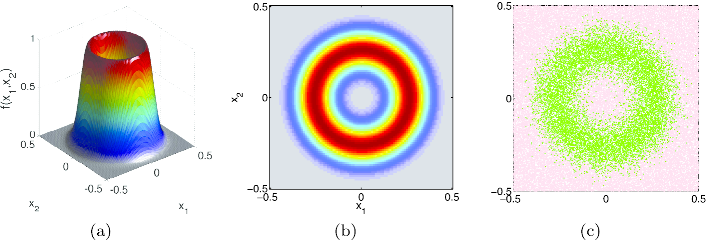}

\caption{(a) The graph of the conditional probability on the unit
square. (b) The level sets of the image of the conditional probability.
(c) For a set of points drawn from the marginal distribution on the
unit square we label them red or green based on the conditional
probability given by \protect\eqref{eq:ex_mu}. The green points are
those assigned to a response of $one$ and the red points are those
assigned zeros.}\vspace*{6pt}\label{fig:mu}
\end{figure}

%%%%%%%%%%

%%%%%%%%%%

We generate $\iid$ observations $\{(X_1,Y_1), \ldots, (X_n, Y_n)\}$
from the joint distribution and our objective is to recover the
topology of the level set $D_L$ for $L = 0.5$ which is used as the
binary classifier in this case, and has the shape of an annulus. We use
the Rips construction presented in Theorem~\ref{teo:persistence_rips},
with $n={}$50,000, $r = 0.01$, and $\eps= 0.2$. This gives us two
complexes: $S_1 = R_{0.3}(n,r)$ and $S_2 = R_{0.7}(n,r)$. Figure~\ref
{fig:filtering} shows the sets of disks used to create the two Rips
complexes. The light blue disks are the ones corresponding to $S_1$ and
the orange ones corresponds to $S_2$. Computing the Betti numbers yields: %(see Table~\ref{tab1}):

\begin{center}\vspace*{9pt}
%\tablewidth=200pt
%\tabcolsep=0pt
%\caption{\qq{}}\label{tab1}
\begin{tabular*}{200pt}{@{\extracolsep{\fill}}@{}lccc@{}}
\hline
&$\bolds{S_1}$ & $\bolds{S_2}$ & $\bolds{S_1\hookrightarrow S_2}$\\
\hline
$\beta_0$ & 34 & 53 & 1 \\
$\beta_1$ & 23 & 49 & 1 \\
\hline
\end{tabular*}\vspace*{9pt}
\end{center}

Indeed, while the homology of each of the complexes $S_1,S_2$ is
extremely noisy, the image of the map between them looks exactly like
an annulus.

%s5.2 #&#
\subsection{Kernel regression}

%%%%%%%%%%

In this example, we consider a regression function on the unit square
$f:[-1,1]^2\to\R$ with additive noise
%
%
%e5.2 #&#
\begin{equation}
\label{eq:reg_additive} Y_i = f(X_i) + \xi_i.
\end{equation}
Our objective will be to recover the barcode or persistent homology of
the above function from noisy observations.

%
%f9 #&#
\begin{figure}[t]

\includegraphics{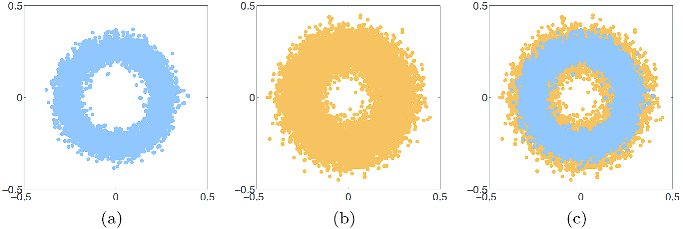}

\caption{Computing the homology of a level set for a regression
function. We generated $\{(X_i,Y_i)\}_{i=1}^{50,000}$ $\iid$
observations from the marginal and conditional distributions given in
equation \protect\eqref{eq:ex_mu}. For $L=0.5$ and $\eps=0.2$ we
present the following: (a) the set $\hat D_{L+\eps}(n,r)$, (b) the set
$\hat D_{L-\eps}(n,r)$, (c) the two sets combined. Note that both
individual sets in (a) and (b) contain many connected components and
cycles. However, in (c) we observe that most of these homological
features do not survive the transition. All the extra connected
components in (a) are merged into the large component in (b).
Similarly, all the extra cycles in (a) are filled up in (b).}
\label{fig:filtering}
\end{figure}

%%%%%%%%%%

The regression function $f$ was generated from a random mixture of
Gaussians, and its graph is presented in Figure~\ref{fig:pd_reg}(a).
The ``true'' barcode of the function $f$ is presented in Figure~\ref
{fig:reg_barcode}(a).\vadjust{\goodbreak} This barcode was computed by evaluating $f$
directly on a dense grid and computing the persistent homology of this
discretized version. The independent variables $X_i$ are generated from
a uniform distribution in the box $[-1,1]^2$. The noise $\xi_i$ is
independent of $X_i$, and generated by a normal distribution with
$\sigma= 0.2$ truncated at $5\sigma$ (we require in (C2) for the
response variables to be bounded).
To estimate this barcode, we used $\hPH_*^\eps(f)$ (defined in
Section~\ref{sec:persistent_homology}) with $n=5000, r = 0.1, \eps=
0.001$. The result is presented in Figure~\ref{fig:reg_barcode}(b).

%%%%%%%%%%

%
%f10 #&#
\begin{figure}

\includegraphics{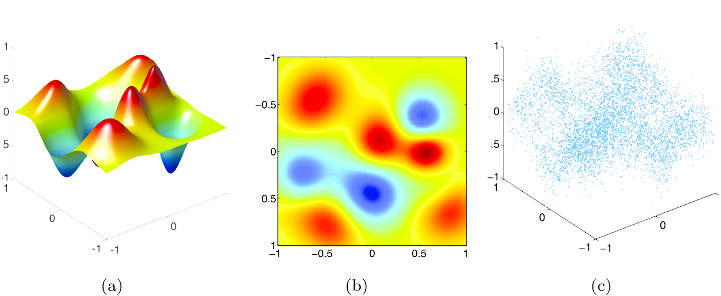}

\caption{A regression function in $\R^2$. (a) The graph of the
function in the box $[-1,1]^2$. (b) The level sets of the function. It
is easy to spot five peaks and three valleys in this image, which in
persistent homology correspond to five features in $\PH_0$ and three
in $\PH_1$.
(c) Generating $\{(X_i,Y_i)\}_{i=1}^{5000}$ $\iid$ observations from
the model presented in \protect\eqref{eq:reg_additive}.}
\label{fig:pd_reg}
\end{figure}

%%%%%%%%%%

%
%f11 #&#
\begin{figure}[b]

\includegraphics{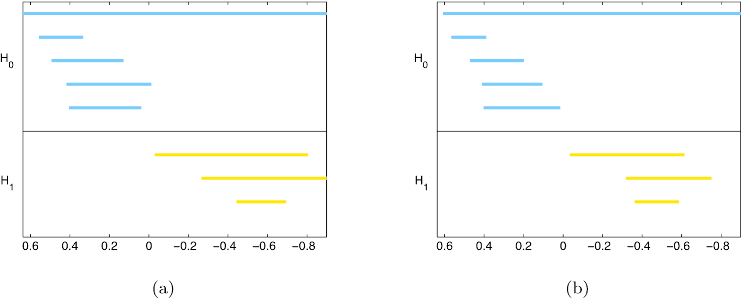}

\caption{(a) The ``true'' barcode of the persistent homology of the
regression function $f$ presented in Figure~\protect\ref{fig:pd_reg}.
(b) The estimated persistent homology $\hPH_*^\eps(f)$, with $n=5000,
r = 0.1, \eps= 0.001$, is very close to the true barcode. For
visualization purposes, we left bars with length less than $0.05$ out
of the figure. In both the true and the estimated barcodes we observe
five significant features in $H_0$ and three in $H_1$, corresponding to
the five peaks and three valleys in the graph of the function $f$.}
\label{fig:reg_barcode}
\end{figure}

%%%%%%%%%%

%%%%%%%%%%

%
%f12 #&#
\begin{figure}

\includegraphics{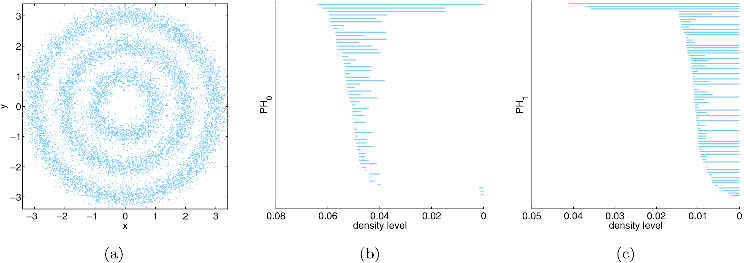}

\caption{(a) A sample set generated from three concentric circles. (b)
The barcode for $\hPH_0^\eps(f)$, where we indeed observe three
dominating components. (c) The barcode for $\hPH_1^\eps(f)$, where we
indeed observe three dominating cycles. The parameters used in this
simulation are $n={}$10,000, $r = 0.125$, $\eps= 0.005$.}
\label{fig:3rings}
\end{figure}

%%%%%%%%%%

%s5.3 #&#
\subsection{Dataset related to spectral clustering}

%%%%%%%%%%

Spectral clustering uses spectral graph theory to cluster observations
(see the review papers in \cite
{ngspectral2002,vonluxburgtutorial2007}). It is mostly useful in
cases where the clusters are not necessarily concentrated close to a
single point, but have a more complicated shape (such as the data in
Figure~\ref{fig:3rings}). We revisit a simulated example from the
spectral clustering literature to illustrate how well we can recover
the number of clusters and cluster features using our level sets
approach. We generate $n={}$10,000 points from three concentric circles
(of radii $1,2,3$) and added multivariate Gaussian noise with $\sigma=
0.2$. The result is presented in Figure~\ref{fig:3rings}(a). The
topological features we wish to recover here are the three connected
components and the three cycles (spectral clustering would find the
three connected components). The parameters we used are $r=0.125, \eps
= 0.005$. Figure~\ref{fig:3rings}(b) displays $\hPH^\eps_0(f)$. Here
we see that there are indeed three dominating features (bars that
persist over a long period of time). The rest of the features are
generated by the fluctuations in the estimated density function.
Similarly, in Figure~\ref{fig:3rings}(c) we observe three dominating
features as well, representing the three cycles in the data.

%%%%%%%%%%

%
%f13 #&#
\begin{figure} 

\includegraphics{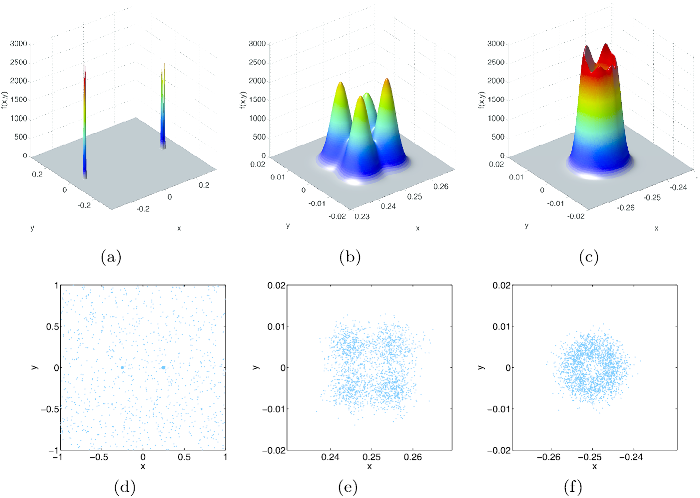}

\caption{A hierarchical density function. (a) The density function at a
coarse level, consisting of two sharp peaks. (b) Zooming in on the
density around $(0.25,0)$ we observe that this sharp peak actually
consists of four adjacent peaks. (c) Zooming in on the density around
$(-0.25,0)$ we observe that the peak has a crater-like structure with
small fluctuation around the rim. (d)--(f) A sample of $n=5000$ points
generated by $f$.}
\label{fig:density}
\end{figure}

%%%%%%%%%%

%s5.4 #&#
\subsection{Hierarchical clustering}

%%%%%%%%%%

This example will be used to show how using our method we can capture
features of a density function with hierarchical structure. Consider a
probability density $f$ on $\R^2$ that consists of two concentrated
densities that are far apart and centered at $(\pm0.25, 0)$, see
Figure~\ref{fig:density}(a). Once we zoom into the two densities we
realize there is a finer structure in this problem. The density around
$(0.25, 0)$ is a mixture of four Gaussians that are very near each
other, see Figure~\ref{fig:density}(b). The density around $(-0.25,
0)$ is one density that looks like a volcano crater (made of a mixture
of $100$ Gaussians), see Figure~\ref{fig:density}(c). The result of
this finer structure is that when we examine the persistence homology
of $f$ we expect to see: (1) five dominating features in $\PH_0$ -- the
four bumps on the right, and the entire volcano on the left, (2) two
dominating features in $\PH_1$ -- one coming from the cycle along the
rim of the volcano, and another one from the cycle that surrounds the
four bumps, (3) fluctuations on the rim will introduce features in
$\PH_0(f)$ but these will have low persistence. We will show how we
can accurately capture the homology of this hierarchical structure.

The barcode in Figure~\ref{fig:pd_multi_res}(a) displays the ``true''
persistent homology $\PH_*(f)$ that was computed by evaluating the
function values directly on a very fine grid around the peaks. Looking
at the barcode of $\PH_0(f)$, we see two dominant features, with death
time close to zero. These two features correspond to the two clusters
represented by the peaks seen in Figure~\ref{fig:density}(a). The
other three dominant features correspond to the three additional peaks
we have in Figure~\ref{fig:density}(b). The rest of the bars (as well
as other shorter bars we kept out of the figure for visualization
purposes) correspond to the fluctuation along the rim of the crater in
Figure~\ref{fig:density}(c). In $\PH_1(f)$, we see exactly two
features corresponding to the two cycles described above.

We can compare the true barcode to the barcode generated by our
estimator for $\PH_*(f)$ using $\hPH_*^\eps(f)$. The parameters\vspace*{1pt} we
used in the estimator are $n=5000$, $r=0.001$, $\eps=3.5$. The barcode
for $\hPH_*^\eps(f)$ is presented in Figure~\ref
{fig:pd_multi_res}(b). The global picture is very consistent with that
of the true function. As expected our estimates have extra variation in
the endpoints of the bars.

%
%f14 #&#
\begin{figure}

\includegraphics{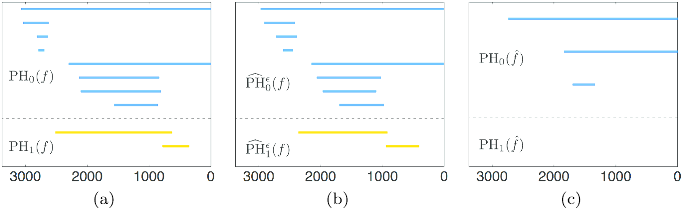}

\caption{Estimating the persistent homology of the density function
$f$ presented in Figure~\protect\ref{fig:density}.
(a) The ``true'' barcode for the function $f$, that is, $\PH_*(f)$
(computed by sampling the
density function on a fine grid). (b) The barcode computed from the
estimator $\hPH_*^\eps(f)$.
The parameters used are $n=5000$, $r=0.001$, $\eps=3.5$. (c) The barcode
computed for the
kernel density estimator -- $\PH_*(\hat f)$. The kernel parameters are
the same as for
$\hPH_*^\eps(f)$, the grid size taken is $500\times500$. Note that
the estimator
$\hPH_*^\eps(f)$ gives a result that is very similar to the true
barcode. In both
cases there are five significant features in $H_0$ and two significant features
in $H_1$. The barcode for $\PH_*(\hat f)$ only recover the coarse
features, namely the two clusters, but completely ignores the finer
structures. We note that for visualization purposes we filtered out the
very small bars before drawing the barcodes here.}
\label{fig:pd_multi_res}
\end{figure}

%%%%%%%%%%

In Fasy \textit{et al}. \cite{fasyconfidence2014}, an alternate approach is
developed to estimate $\PH_*(f)$. Their idea is to use a kernel
density estimation to obtain an estimate $\hat{f}_n$ of the density
$f$. Then they compute the persistent homology of $\hat{f}_n$, denoted
by $\PH_*(\hat{f}_n)$. They are able to provide a theoretical bound
on the bottleneck distance between $\PH_*(f)$ and $\PH_*(\hat
{f}_n)$. This result is similar in spirit to Theorem~\ref
{teo:ph_consistency} in our paper. The main difference in their method
versus our method is that they focus on getting a good\vspace*{1pt} estimate of the
function values or ensuring $\hat{f}_n \approx f(x)$ everywhere,
whereas we compute $\hPH_*^\eps(f)$ by approximating the level sets directly.

In the case of a density function with hierarchical structure, these
two approaches often have different empirical performance. In
particular, we argue that the estimator $\hPH_*^\eps(f)$ is favorable
to $\PH_*(\hat f)$. The crux of the argument in favor of computing
$\hPH_*^\eps(f)$
is that in evaluating the fit of $\hat f$ there is a resolution
parameter of how fine in $\R^2$ one measures $f$, which we denote as~$\Delta$
(in addition to the bandwidth parameter of the kernel -- $r$).
The problem arises in that one needs to know what value of $\Delta$ is
small enough to capture fine structure in $f$. This raises two issues:
(1)~how to adaptively estimate $\Delta$ from data and (2) taking a finer
resolution parameter will result in an increase in the sample
complexity of the inference problem. Our approach of directly
estimating $\hPH_*^\eps(f)$ avoids these difficulties, since we only
work with the original sample points rather than $\hat f$. In
Figure~\ref{fig:pd_multi_res}(c), we present the barcode for $\PH
_*(\hat f)$, computed using the same kernel, on a grid of size
$500\times500$ (i.e.,~$\Delta= 1/250$).

%%%%%%%%%%

%s6 #&#
\section{Conclusion}

%%%%%%%%%%

In this paper, we introduce a consistent estimator for the homology of
level sets for both density and regression functions. We
apply this procedure to infer the homology of a manifold from noisy
observations, and infer the persistent homology of either density or
regression functions. The conditions we require are weaker than
previous results in this
direction. % we also provide consistency under a wider variety of noise
%models.

We view this work as an important step in closing the gap between
topological data analysis and statistics. For topological data
analysis, we provide a consistent estimator for the homology and
persistent homology of spaces underlying random data. As future work,
we will consider refinements of our analysis to obtain convergence
rates and confidence intervals of the estimates. We suspect this will
require more assumptions on the geometry of the underlying spaces. From
a statistical perspective, this work suggests that topological
summaries of density and regression functions are of interest and
provide insights in statistical modeling. We suspect these
characteristics or topological summaries will be very useful in
classification or hypothesis testing problems, when the assumptions on
different decision regions can be naturally captured by coarse geometry
or topology.

%%%%%%%%%%

\begin{appendix}\label{sec:proofs}\label{append}
%s7 #&#
\section*{Appendix: Proofs}

%%%%%%%%%%

In this section we provide the proofs for Theorems
\ref{teo:persistence},~\ref{teo:manires},~\ref{teo:ph_consistency},
and~\ref{teo:persistence_rips}.

%%%%%%%%%%

%s7.1 #&#
\subsection{Some definitions and lemmas}

%%%%%%%%%%

Recall that
\[
\cX_n^L \triangleq \bigl\{X_i: \hat
f_n(X_i) \ge L; 1\le i \le n \bigr\}. %
\]
Our first\vspace*{1pt} step would be to assign some probabilistic quantification of
the accuracy of the assignments $\cX_n^L$ with respect to $D_L$. We
will do
this by first defining two sets: the set $D^{\uparrow}_{L,r}$
corresponds to
``inflating'' $D_L$ by a radius $r$ and $D^{\downarrow}_{L,r}$
corresponds to
``deflating'' $D_L$ by a radius $r$. To define these sets, we first
define the tube of radius $r$ around the boundary of $D_L$
\[
\partial D_{L}^{r} = \bigcup
_{x\in\partial D_L} B_r(x), \qquad \partial D_L\mbox{ is the boundary of } D_L.
\]
We then define $D^{\uparrow}_{L,r}$ and $D^{\downarrow}_{L,r}$ as follows
\[
D^{\uparrow}_{L}(r) = D_L \cup\partial
D_{L}^r, \qquad D^{\downarrow
}_L(r) =
D_L\setminus\partial D_{L}^{r}.
\]

Using these definitions the following lemma provides a bound on the
false positive and false negative error of the set $\cX_n^L$
with respect to $D_L$.

%%%%%%%%%%

%
%le7.1 #&#
\begin{lem} \label{lem:prob_filtering} Assume that constraint \textup{(C1)} on
the kernel function holds and either condition \textup{(C2)} holds for the
regression case or in the density estimation case the density is
bounded and tame. For every $L>0$, and $\eps\in(0,L)$, if $r\to0$
and $nr^d\to\infty$, then there exists a constant $C^\star_\eps$
such that for $n$ large enough we have
%
%
%e7.1 #&#
\begin{equation}
\label{eq:prob_filtering_lower} \mathbb{P} \bigl(\exists X_i \notin
D^{\uparrow}_{L-\eps}(r): \hat f_n(X_i) \ge L
\bigr) \le n e^{- C^\star_\eps nr^d},
\end{equation}
and
%
%
%e7.2 #&#
\begin{equation}
\label{eq:prob_filtering_upper} \mathbb{P} \bigl(\exists X_i \in
D^{\downarrow}_{L+\eps}(r): \hat f_n(X_i) \le L
\bigr) \le n e^{-C^\star_\eps nr^d}.
\end{equation}
\end{lem}

%%%%%%%%%%

Equation \eqref{eq:prob_filtering_lower} bounds the probability of a
false-positive error, and equation \eqref{eq:prob_filtering_upper}
bounds the probability of a false-negative error. The value of $C^\star
_\eps$ is different for density estimation versus regression
and is given by \eqref{eq:C_star_density} and \eqref{eq:C_star_regression}.

Next, recall that
\[
\hat D_L(n,r) \triangleq U\bigl(\cX_n^L, r
\bigr).
\]
We would like to prove that with a high probability this empirical set
is sandwiched by two sets which should be ``close'' to $D_L$. The
following lemma states the precise result.

%%%%%%%%%%

%
%le7.2 #&#
\begin{lem}\label{lem:set_bounds}
For every $L>0$, and $\eps\in(0,L)$, if $r\to0$ and $nr^d\to\infty
$, then for large enough $n$ we have
\[
\mathbb{P} \bigl(D^{\downarrow}_{L+\eps}(2r) \subset\hat
D_L(n,r) \subset D^{\uparrow}_{L-\eps }(2r) \bigr) \ge1 - 3
ne^{-C^\star
_\eps nr^d},
\]
\end{lem}

%%%%%%%%%%

In other words, our estimator $\hat D_L(n,r)$ is sandwiched between the
two non-random approximations of $D_L$.

The last ingredient we need for the proving the theorems is the
following purely algebraic lemma.

%%%%%%%%%%

%
%le7.3 #&#
\begin{lem}\label{lem:algebra}
Consider the following commutative diagram of groups,
\[
\mbox{
\includegraphics{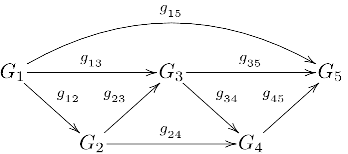}
}
%\xymatrix@1{ G_1 \ar@/^2pc/[rrrr]^{g_{{15}}}
%\ar[dr]^{g_{{12}}}\ar [rr]^{g_{{13}}} & &G_3
%\ar[rr]^{g_{{35}}} \ar[dr]^{g_{{34}}}& & G_5
%\\
%& G_2 \ar[rr]^{g_{{24}}} \ar[ur]^{g_{{23}}} & &
%G_4 \ar[ur]^{g_{{45}}}}
\]
(by ``commutative'' we mean that all paths with the same endpoints lead
to the same result),
and for every $i,j$ define $G_{ij} = \im(g_{{ij}}) \subset G_j$.

If $g_{{35}}:G_3\to G_5$ is an isomorphism from $G_3$ to $G_{15}$.
Then the map $g_{{34}}:G_3\to G_4$ is an isomorphism from $G_3$ to
$G_{24}\subset G_4$. In particular, we have $G_3\cong G_{24}$.
\end{lem}

%%%%%%%%%%

%s7.2 #&#
\subsection{Proving the theorems}

%%%%%%%%%%

\begin{pf*}{Proof of Theorem~\ref{teo:persistence}}
Using Lemma~\ref{lem:set_bounds} for $\hat D_{L+\eps}(n,r)$ and $\hat
D_{L-\eps}(n,r)$ we have that for $n$ large enough
%
%
%e7.3 #&#
%e7.4 #&#
\begin{eqnarray}
\label{eq:persistence_inclusion} %
\mathbb{P} \bigl(D^{\downarrow}_{L + \sklfrac{3}{2}\eps}(2r)
\subset \hat D_{L+\eps}(n,r) \subset D^{\uparrow}_{L+\sklfrac{1}{2}\eps}(2r)
\bigr) \ge1 - 3ne^{-C^\star_{\eps
/2} nr^d},
\nonumber\\[-8pt]\\[-8pt]\nonumber
\mathbb{P} \bigl(D^{\downarrow}_{L - \frac{1}{2}\eps}(2r)\subset \hat
D_{L-\eps}(n,r) \subset D^{\uparrow}_{L-\sklfrac{3}{2}\eps}(2r) \bigr) \ge1 -
3ne^{-C^\star_{\eps
/2} nr^d}.
\end{eqnarray}
Since we assume $L$ is $\eps$-regular, if $r$ is small enough, we have
\[
D_{L+2\eps} \subset D^{\downarrow}_{L+\sklfrac{3}{2}\eps}(2r) \subset
D^{\uparrow} _{L+\sklfrac{1}{2}\eps}(2r) \subset D_L \subset
D^{\downarrow}_{L-\sklfrac
{1}{2}\eps
}(2r) \subset D^{\uparrow}_{L-\sklfrac{3}{2}\eps}(2r)
\subset D_{L-2\eps},
\]
and from \eqref{eq:persistence_inclusion} we conclude that
\[
\mathbb{P} \bigl(D_{L+2\eps} \subset\hat D_{L+\eps}(n,r)
\subset D_L \subset \hat D_{L-\eps}(n,r) \subset
D_{L-2\eps} \bigr) \ge1- 6ne^{-C^\star_{\eps
/2}nr^d}.
\]
Denote
\[
S_1 = D_{L+2\eps},\qquad
S_2= \hat D_{L+\eps}(n,r),\qquad
S_3 =D_L,\qquad
S_4 = \hat D_{L-\eps}(n,r),\qquad
S_5= D_{L-2\eps},
\]
and let $G_i = H_*(S_i)$.
Since we assume that $f(x)$ has no critical values in $[L-2\eps,
L+2\eps]$,
and using the notation of Lemma~\ref{lem:algebra} we have that the
maps $g_{{13}}, g_{{35}}$ and $g_{{15}}$ induced by the inclusions
$S_1\subset S_3 \subset S_5$ are all isomorphisms.
If, in addition, $S_1\subset S_2\subset S_3 \subset S_4 \subset S_5$,
then using Lemma~\ref{lem:algebra} we conclude that $G_{24} \cong
G_3$. Observing that $G_{24} = \im(\imath_*)$ (see \eqref
{eq:i_star}) we have that $\im(\imath_*) \cong H_*(D_L)$ which
completes the proof.
\end{pf*}

%%%%%%%%%%

\begin{pf*}{Proof of Theorem~\ref{teo:manires}}
Recall that $N_\eps= \sup_{x\in\R^d}  \lceil{f(x)/2\eps
} \rceil$, and
$L_{\max} = 2\eps N_\eps$.
Let $E$ be the event that for every $1\le i \le N_\eps$ the following
inclusion holds:
%
%
%e7.5 #&#
\begin{equation}
\label{eq:inc_L_i} D_{L_{i-1}} = D_{L_i+2\eps} \hookrightarrow\hat
D_{L_i+\eps}(n,r) \hookrightarrow D_{L_i} \hookrightarrow\hat
D_{L_i-\eps}(n,r) \hookrightarrow D_{L_i-2\eps}= D_{L_{i+1}}.
\end{equation}
Applying\vspace*{1pt} Lemma~\ref{lem:set_bounds} (as in the proof of Theorem~\ref
{teo:persistence}) $N_\eps$ times we can show that if $r\to0$ and
$nr^d\to\infty$ then for $n$ large enough
\[
\mathbb{P} (E ) \ge1- 3n N_\eps e^{-C^\star_{\eps/2} nr^d}.
\]
From here on we will assume that \eqref{eq:inc_L_i} is true for all
$1\le i \le N_\eps$.
Choosing $i^\star$ as
\[
i^\star\triangleq1+\min \bigl\{i\in \{1,\ldots,N_\eps \}:\hat
\beta_m(L_i,\eps;n) = 1 \bigr\},
\]
our goal is to show that $[L_{i^\star}-2\eps, L_{i^\star}+2\eps]
\subset(A,B)$, and therefore the arguments used in the proof of
Theorem~\ref{teo:persistence} guarantee that $\hat H_*(L_{i^\star
},\eps;n) \cong H_*( D_{L_{i^\star}}) \cong H_*(\cM)$.

Since $\cM$ is assumed to be connected, we have that $\beta_0(\cM
)=1$, and by Poincar\'{e} duality (cf. \cite
{hatcheralgebraic2002,munkreselements1984}) we conclude that $\beta
_m(\cM) = 1$. If $L_i\in(A,B)$ then from Definition~\ref
{def:noisy_dens} we have that $D_{L_i}\simeq\cM$ and thus $\beta
_m(D_{L_i}) = 1$ as well.
On the other hand, if $L_i > B$ then $D_{L_i}\simeq\cM'$ where $\cM
'$ is a compact locally contractible proper subset of the $\cM$.
Using Proposition 3.46 in \cite{hatcheralgebraic2002} we have that
$\beta_m(\cM') = \beta_m(D_{L_i}) = 0$.

Our requirement that $L_{i-1}-L_i = 2\eps$ and $B-A \ge8\eps$
guarantees that there are at least four consecutive levels $L_i$ such
that $L_i\in(A,B)$. Let $L_{i_1} > L_{i_2} > L_{i_3} >L_{i_4}$ be the
first (highest) such levels. For $k=2,3$ we have that $[L_{i_k}-2\eps,
L_{i_k}+2\eps] \subset(A,B)$, and from the proof of Theorem~\ref
{teo:persistence} and the previous paragraph we conclude that $\hat
\beta_m(L_{i_k},\eps;n) = 1$. For $i_1$ however, it is not true that
$[L_{i_1}-2\eps, L_{i_1}+2\eps] \subset(A,B)$ and therefore, $\hat
\beta_m(L_{i_1},\eps;n)$ might be either zero or one.
Finally, defining $i^\star$ the way we did, $i^\star$ might be either
$i_2$ or $i_3$. In both cases we have $[L_{i^\star}-2\eps, L_{i^\star
}+2\eps] \subset(A,B)$, and that completes the proof.
\end{pf*}

%%%%%%%%%%

\begin{pf*}{Proof of Theorem~\ref{teo:ph_consistency}}
Recall that $\cD=  \{D_L \}_{L\in\R}$ is the continuous filtration
of the (super) level sets of $f$, and $\hat{\cD}^\eps= \{ \hat
D_{L_i}(n,r)\}_{i\in\Z}$ is a discrete approximation.\vspace*{1pt} To prove that
the corresponding persistent homologies $\PH_*(f), \hPH_*^\eps(f)$ satisfy
\[
d_B \bigl(\hPH_*^\eps(f), {\PH}_*(f) \bigr) \le5\eps,
\]
we will use the language of $\eps$-interleaving introduced in \cite
{chazalproximity2009}.
The first step would be to define a discrete version of the filtration
$\cD$ given by
\[
\cD^\eps\triangleq \{D_{L_i+\eps} \}_{i\in\Z},
\]
where $L_i$ is defined in \eqref{eq:levels}. Denote the persistent
homology of $\cD^\eps$ by $\PH_*^\eps(f)$. Since $\cD^\eps$ is a
discrete approximation of the continuous filtration $\cD$, with step
size $2\eps$, the maximum difference between $\PH_*(f)$ and $\PH
_*^\eps(f)$ would be the step size, and thus we have
\[
\label{eq:bn_distance_cont} d_B\bigl(\PH_*^\eps(f), \PH_*(f)\bigr)
\le2\eps.
\]
To prove the theorem, it is therefore enough to show that with a high
probability we have $d_B(\hPH_*^\eps(f), \PH_*^\eps(f))\le3\eps$.

Let $E$ be the event that we have the following sequence of inclusions:
%
%
%e7.6 #&#
\begin{equation}
\label{eq:interleave}\small %
\begin{array} {ccccccccccccc} D_{L_0+\eps} &
& & & D_{L_1+\eps} & & & & D_{L_2+\eps} & & & & \cdots
\\
& \hookurarrow& & \hookdrarrow& & \hookurarrow& & \hookdrarrow& & \hookurarrow&
& \hookdrarrow&
\\
& & \hat D_{L_0}(n,r) & & & & \hat D_{L_1}(n,r) & & & & \hat
D_{L_2}(n,r) & & \end{array}.
\end{equation}
Applying Lemma~\ref{lem:set_bounds} $N_\eps$ times we can show that
if $n$ is large enough
\[
\mathbb{P} (E ) \ge1- 3 nN_\eps e^{-C^\star_{\eps/2}nr^d}.
\]
Using the notation in \cite{chazalproximity2009} \eqref
{eq:interleave} implies that $\cD^\eps$ and $\hat{\cD}^\eps$ are
\emph{weakly $\eps$-interleaving}.
Denoting the persistent homology of $\hat{\cD}^\eps$ by $\hPH
_*^\eps(f)$, using Theorem 4.3 in \cite{chazalproximity2009} yields
%
%
%e7.7 #&#
\begin{equation}
\label{eq:ph_estimate_dist} d_B \bigl(\hPH_*^\eps(f), {
\PH}_*^\eps(f) \bigr)\le3\eps.
\end{equation}
This completes the proof.
\end{pf*}

%%%%%%%%%%

\begin{pf*}{Proof of Theorem~\ref{teo:persistence_rips}}
Consider the following sequence of simplicial complexes,
\[
C_{L\pm\eps}(n,r) \hookrightarrow R_{L\pm\eps}(n,r) \hookrightarrow
C_{L\pm\eps}(n, \sqrt{2}r).
\]
This sequence induces the following sequence in homology
\[
H_*\bigl(C_{L\pm\eps}(n,r)\bigr) \to H_*\bigl( R_{L\pm\eps}(n,r)\bigr)
\to H_*\bigl(C_{L\pm
\eps}(n, \sqrt{2}r)\bigr),
\]
or equivalently,
%
%
%e7.8 #&#
\begin{equation}
\label{eq:cech_rips_inc} H_*\bigl( \hat D_{L\pm\eps}(n,r)\bigr) \to H_*\bigl(
R_{L\pm\eps}(n,r)\bigr) \to H_*\bigl( \hat D_{L\pm\eps}(n,\sqrt{2}r)
\bigr).
\end{equation}
From the proof of Lemma~\ref{lem:set_bounds} (see \eqref
{eq:prob_down_bound},\eqref{eq:prob_up_bound}) we have that
\begin{eqnarray*}
\mathbb{P} \bigl(D^{\downarrow}_{L+\sklfrac{3}{2}\eps}(2r) \not \subset
\hat D_{L+\eps }(n,r) \bigr) &\le&2ne^{-C^\star_{\eps/2}
nr^d},
\\
\mathbb{P} \bigl(D^{\downarrow}_{L-\sklfrac{1}{2}\eps}( 2r) \not \subset
\hat D_{L-\eps }(n,r) \bigr) &\le& 2ne^{-C^\star_{\eps/2}
nr^d},
\\
\mathbb{P} \bigl(\hat D_{L+\eps}(n,\sqrt{2} r) \not\subset
D^{\uparrow}_{L+ \sklfrac {1}{2}\eps}( 2\sqrt{2}r) \bigr) &\le& ne^{-C^\star_{\eps/2} 2^{d/2}nr^d},
\\
\mathbb{P} \bigl(\hat D_{L-\eps}(n,\sqrt{2}r) \not\subset
D^{\uparrow}_{L- \sklfrac {3}{2}\eps}( 2\sqrt{2}r) \bigr) &\le& ne^{-C^\star_{\eps/2} 2^{d/2}nr^d}.
\end{eqnarray*}
Therefore, for $n$ large enough we have
\begin{eqnarray*}
\mathbb{P} \bigl(D^{\downarrow}_{L+\sklfrac{3}{2}\eps}(2r) \subset \hat
D_{L+\eps }(n,r)\subset\hat D_{L+\eps}(n,\sqrt{2}r)\subset
D^{\uparrow}_{L+ \sklfrac {1}{2}\eps}( 2\sqrt{2}r) \bigr) &\ge& 1-3ne^{C^\star_{\eps/2} nr^d},
\\
\mathbb{P} \bigl(D^{\downarrow}_{L-\sklfrac{1}{2}\eps}( 2r) \subset \hat
D_{L-\eps }(n,r)\subset\hat D_{L-\eps}(n,\sqrt{2}r)\subset
D^{\uparrow}_{L- \sklfrac {3}{2}\eps}( 2\sqrt{2}r) \bigr) &\ge& 1-3ne^{C^\star_{\eps/2} nr^d}.
\end{eqnarray*}
Since we assume that all the levels we study are $\eps$-regular, if
$r$ is small enough we can order them in the following way
\begin{eqnarray*}
D_{L+2\eps} &\subset& D^{\downarrow}_{L+\sklfrac{3}{2}\eps}(2r) \subset
D^{\uparrow} _{L+\sklfrac{1}{2}\eps}(2\sqrt{2}r) \subset D_L \subset
D^{\downarrow
}_{L-\sklfrac
{1}{2}\eps}(2r)
\\
&\subset& D^{\uparrow}_{L-\sklfrac{3}{2}\eps}(2
\sqrt{2}r) \subset D_{L-2\eps}.
\end{eqnarray*}
Combining that with \eqref{eq:cech_rips_inc}, we conclude that with a
high probability we have the following sequence in homology (induced by
composing inclusion maps),
\[
\begin{array} {ccccccc} \bigstar& H_*(D_{L+2\eps})&
\rightarrow&H_*\bigl(D^{\downarrow}_{L+\sklfrac
{3}{2}\eps
}(2r)\bigr) & \rightarrow& H_*
\bigl(\hat D_{L+\eps}(n,r)\bigr) &
\\
&&&&& \downarrow&
\\
&&&&& H_*\bigl(R_{L+\eps}(n,r)\bigr)& \bigstar
\\
&&&&& \downarrow&
\\
&&\leftarrow&H_*\bigl(D^{\uparrow}_{L+ \sklfrac{1}{2}\eps}( 2\sqrt {2}r)\bigr)&
\leftarrow & H_*\bigl(\hat D_{L+\eps}(n,\sqrt{2}r)\bigr)&
\\
\bigstar& H_*(D_L) &&&&&
\\
&&\rightarrow& H_*\bigl(D^{\downarrow}_{L-\sklfrac{1}{2}\eps
}(2r)\bigr)&\rightarrow&
H_*\bigl(\hat D_{L-\eps}(n,r)\bigr) &
\\
&&&&&\downarrow&
\\
&&&&& H_*\bigl(R_{L-\eps}(n,r)\bigr) & \bigstar
\\
&&&&&\downarrow&
\\
\bigstar& H_*(D_{L-2\eps}) &\leftarrow&H_*\bigl(D^{\uparrow}_{L- \sklfrac
{3}{2}\eps
}(
2\sqrt{2}r)\bigr)&\leftarrow& H_*\bigl(\hat D_{L-\eps}(n,\sqrt{2}r)\bigr) &

\end{array} %
\]
Taking out the spaces marked in $\bigstar$ we have
\[
H_*(D_{L+2\eps})\to H_*\bigl(R_{L+\eps}(n,r)\bigr) \to
H_*(D_L) \to H_*\bigl(R_{L-\eps}(n,r)\bigr) \to
H_*(D_{L-2\eps}).
\]
Since $f(x)$ has no critical values in $[L-2\eps, L+2\eps]$, using
Lemma~\ref{lem:algebra} completes the proof.
\end{pf*}

%%%%%%%%%%

%s7.3 #&#
\subsection{Proving the lemmas}
One of the main probability tools we use is Bernstein's inequality
\cite{RAP}, basically a law of large numbers bound.
If $Z_1,\ldots, Z_n$ are $\iid$, with $\E \{{Z_i} \} = 0,
\Var(Z_i) =
\sigma^2$ such that $\llvert {Z_i}\rrvert \le M$ almost
surely, then
%
%
%e7.9 #&#
\begin{equation}
\mathbb{P} \Biggl(\sum_{i=1}^n
Z_i \ge t \Biggr) \le\exp \biggl(-\frac{t^2/2}{n\sigma^2+Mt/3}
\biggr)\label
{eq:bernstein}.
\end{equation}

%%%%%%%%%%

\begin{pf*}{Proof of Lemma~\ref{lem:prob_filtering} (Density estimation)}
To reconstruct the level sets of the density, we will use a kernel
density estimator.
Recall that the kernel function $K:\R^d\to\R$ we use satisfies the following:
\begin{itemize}
\item$\supp(K) \subset B_1(0)$,
\item$K(x) \in[0,1]$, and $K(0) = 1$,
\item$\int K(\xi)\,d\xi= C_K$, for some $C_K \in(0,1)$.
\end{itemize}
In this case, our kernel estimator is
\[
\hat f_n(x) = \frac{\sum_{i=1}^n K_r(x-X_i)}{C_K n r^d},
\]
where $K_r(x) = K(x/r)$.
We start by proving \eqref{eq:prob_filtering_lower}.
Using a simple union bound we have
%
%
%e7.10 #&#
\begin{eqnarray}\label{eq:density_total_prob_lower}
\nonumber
\mathbb{P} \bigl(\exists X_i \notin
D^{\uparrow}_{L-\eps}(r): \hat f_n(X_i) \ge L
\bigr) &\le& n \mathbb{P} \bigl(X_1 \in
\bigl(D^{\uparrow}_{L-\eps}(r)\bigr)^c: \hat
f_n(X_1) \ge L \bigr)
\nonumber\\[-8pt]\\[-8pt]\nonumber
&=& n \int_{(D^{\uparrow}_{L-\eps}(r))^c} f_X(x) \mathbb{P} \bigl(\hat
f_n(X_1) \ge L\given X_1=x \bigr)\,dx.
\end{eqnarray}
Next,
%
%
%e7.11 #&#
\begin{eqnarray}\label{eq:prob_density}
\nonumber
\mathbb{P} \bigl(\hat f_n(X_1) \ge L\given
X_1=x \bigr) &=& \mathbb{P} \Biggl(K_r(0) +
\sum_{i=2}^n K_r(x-X_i)
\ge L C_Knr^d \Biggr)
\nonumber\\[-8pt]\\[-8pt]\nonumber
&=&\mathbb{P} \Biggl(\sum_{i=2}^n
Z_i \ge n\bigl(LC_Kr^d - p_r(x)
\bigr) +p_r(x)-1 \Biggr),
\end{eqnarray}
where
\[
p_r(x) \triangleq\E \bigl\{{K_r(x-X_i)}
\bigr\},
\]
and $Z_i = K_r(x-X_i)-p_r(x)$ are independent variables with $\E
\{{Z_i} \} = 0$.
Note that $p_r(x)\in[0,1]$ since $K_r(x)\in[0,1]$. Also, since $x\in
(D^{\uparrow}_{L-\eps}(r))^c$, we have that
%
%
%e7.12 #&#
\begin{equation}
\label{eq:pr_upper_bound} p_r(x) = \int_{B_r(x)} f(
\xi)K_r(x-\xi)\,d\xi\le(L-\eps)C_Kr^d,
\end{equation}
and therefore from \eqref{eq:prob_density} we have,
%
%
%e7.13 #&#
\begin{equation}
\label{eq:prob_density_bound} \mathbb{P} \bigl(\hat f_n(X_1) \ge L
\given X_1=x \bigr) \le\mathbb{P} \Biggl(\sum
_{i=2}^n Z_i \ge\eps
C_Knr^d -1 \Biggr).
\end{equation}
We would like to apply the inequality in \eqref{eq:bernstein} for $t
=\eps C_Kn r^d -1$. Note that $\llvert {Z_i}\rrvert  \le1$,
and also that
\[
\operatorname{Var} ({Z_i} ) \le\E \bigl\{ {K_r^2(x-X_i)}
\bigr\} \le\fmax C_K r^d.
\]
Therefore, we have
\begin{eqnarray*}
\mathbb{P} \bigl(\hat f_n(X_1) \ge L\given
X_1=x \bigr) &\le&\exp \biggl(-\frac{t^2/2}{(n-1) \fmax C_K r^d+ t/3} \biggr)
\\
&=& \exp \biggl(-\frac{t/2}{t^{-1}(n-1)
\fmax C_K r^d+ \sfrac{1}{3}} \biggr).
\end{eqnarray*}
Since $nr^d\to\infty$, we have
\[
\frac{\sklfrac{1}{2}t(nr^d)^{-1}}{t^{-1}(n-1) \fmax C_K r^d+ \sfrac
{1}{3}} \to\frac{3\eps^2C_K}{6f_{\max}+2\eps} > \frac{\eps^2 C_K
}{3{\fmax}+ \eps}.
\]
Thus, for $n$ large enough we have
\[
\mathbb{P} \bigl(\hat f_n(X_1) \ge L\given
X_1=x \bigr) \le e^{-C^\star_\eps nr^d},
\]
where
%
%
%e7.14 #&#
\begin{equation}
\label{eq:c_star_density} C^\star_\eps= \frac{\eps^2 C_K }{3{\fmax}+ \eps}.
\end{equation}
Which completes the proof of \eqref{eq:prob_filtering_lower}

To prove \eqref{eq:prob_filtering_upper} we start the same way, and
similarly to \eqref{eq:prob_density_bound} we have,
\begin{eqnarray}
\nonumber
\mathbb{P} \bigl(\hat f_n(X_1) \le L\given
X_1=x \bigr) \le \mathbb{P} \Biggl(\sum
_{i=2}^n Z_i \le- \eps
C_Kn r^d \Biggr),
\end{eqnarray}
where we used the fact that $x\in D^{\downarrow}_{L+\eps,r}$, and
therefore we
have $(L+\eps)C_K r^d\le p_r(x)\le1$. Thus, to complete the proof we
should use \eqref{eq:bernstein} for the variables $(-Z_i)$ and $t =
\eps C_K nr^d$. Similarly to the proof above, we then have that
\[
\mathbb{P} \bigl(\hat f_n(X_1) \le L\given
X_1=x \bigr) \le e^{-C^\star_\eps n r^d },
\]
which completes the proof.
\end{pf*}

%%%%%%%%%%

\begin{pf*}{Proof of Lemma~\ref{lem:prob_filtering} (Kernel regression)}
Recall that in the kernel regression model, we have a set of pairs
$(X_1,Y_1),\ldots, (X_n,Y_n)$, where the pairs are $\iid$, $X_i \in
\R^d$, $Y_i \in\R$, and they have a common density function
$f_{X,Y}:\R^d\times\R\to\R$. Our estimation target is the
conditional expectation
\[
f(x) = \E \{Y\given X=x \}.
\]
The estimator we use is given by
\[
\hat f_n(x) = \frac{\sum_{i=1}^n Y_i K_r(x-X_i)}{\sum_{i=1}^n K_r(x-X_i)},
\]
where the assumptions on $K_r$ are the same as above. In addition we
have the following assumptions:
\begin{itemize}
\item$f_X$ has a compact support -- $\supp(f)$.
\item$\fmin\triangleq\inf_{x\in\supp(f)} f_X(x)> 0$,
\item$\llvert {Y_i}\rrvert  \le Y_{\max}$ almost surely, for
some non-random
value $\Ymax> 0$.
\end{itemize}

We start by proving \eqref{eq:prob_filtering_lower}. We use the union
bound again to have
%
%
%e7.15 #&#
%e7.16 #&#
\begin{eqnarray}
\label{eq:regression_total_prob_lower} %
&&
\mathbb{P} \bigl(\exists X_i
\notin D^{\uparrow}_{L-\eps}(r): \hat f_n(X_i)
\ge L \bigr)
\nonumber\\[-8pt]\\[-8pt]\nonumber
&& \quad\le n \int_{(D^{\uparrow}_{L-\eps}(r))^c}\int_{\R}
f_{X,Y}(x,y) \mathbb{P} \bigl(\hat f_n(X_1)
\ge L\given X_1=x, Y_1=y \bigr)\,dy\,dx.
\end{eqnarray}
Note that writing $\hat f_n(x) \ge L$ is equivalent to
\[
\sum_{i=1}^n Y_i
K_r(x-X_i) \ge\sum_{i=1}^n
L K_r(x-X_i).
\]
Using the fact that $x \in(D^{\uparrow}_{L-\eps}(r))^c$, similar derivations
to the ones used for density functions can be applied to show that
\begin{eqnarray*}
\mathbb{P} \bigl(\hat f_n(X_1) \ge L\given
X_1=x, Y_1=y \bigr) &\le& \mathbb{P} \Biggl(
\sum_{i=2}^n Z_i \ge\eps(n-1)
p_r(x) + L-y \Biggr)
\\
&\le&\mathbb{P} \Biggl(\sum_{i=2}^n
Z_i \ge\eps f_{\min}C_K(n-1) r^d +
L-y \Biggr),
\end{eqnarray*}
where here
\[
Z_i \triangleq\bigl(Y_i - f(X_i)\bigr)
K_r(x-X_i) - \eps\bigl(K_r(x-X_i)
- p_r(x)\bigr),
\]
and $p_r(x) = \E \{{K_r(x-X_i)} \}$, and we used the fact
that $p_r(x)
\ge\fmin C_K r^d$.
We would like to use Bernstein's inequality to bound this probability.
First, denote
\begin{eqnarray*}
Z_i^{(1)} &=& \bigl(Y_i-f(X_i)
\bigr)K_r(x-X_i),
\\
Z_i^{(2)} &=& \eps\bigl(K_r(x-X_i)
- p_r(x)\bigr).
\end{eqnarray*}
Then it is easy to show that $\E\{{Z_i^{(1)}}\} = \E\{{Z_i^{(2)}}\} =
\E\{{Z_i^{(1)} Z_i^{(2)}}\} = 0$, which implies that $Z_i^{(1)}$ and
$Z_i^{(2)}$ are uncorrelated, and therefore
\[
\sigma^2 = \operatorname{Var} ({Z_i} ) = \operatorname
{Var} \bigl({Z_i^1} \bigr) + \operatorname{Var}
\bigl({Z_i^2} \bigr).
\]
Also, it is easy to show that
\[
\operatorname{Var}\bigl({Z_i^{(1)}}\bigr) = \E \bigl\{{
\operatorname{Var} ({Y_i\mid X_i} ) K_r^2(x-X_i)}
\bigr\}.
\]
Therefore, we have:
\begin{itemize}
\item$\operatorname{Var}({Z_i^{(1)}}) \le\Ymax^2 \E \{
{K_r^2(x-X_i)} \} \le\Ymax
^2C_K\fmax r^d$,
\item$\operatorname{Var}({Z_i^{(2)}}) \le\eps^2 \E \{
{K_r^2(x-X_i)} \} \le\eps^2
C_K \fmax r^d$,
\item and almost surely:
\[
\llvert {Z_i}\rrvert \le\llvert {Y_i}\rrvert + \bigl
\llvert {f(X_i)}\bigr\rrvert + \eps\bigl(1+p_r(x)\bigr)
\le2\Ymax +\eps\bigl(1+C_K\fmax r^d\bigr) < 2(\Ymax+
\eps).
\]
\end{itemize}

Using Bernstein's inequality \eqref{eq:bernstein}, for $t = \eps
f_{\min}C_K(n-1) r^d + L-y$, we then have
\begin{eqnarray*}
&&\mathbb{P} \bigl(\hat f_n(X_1) \ge L\given
X_1=x, Y_1=y \bigr)
\\
&&\quad\le\exp \biggl(-\frac{t/2}{t^{-1}( \Ymax^2 +\eps
^2)C_K\fmax(n-1)r^d +\sklfrac{2}{3}(\Ymax+\eps)} \biggr).
\end{eqnarray*}
Since $nr^d\to\infty$, we have that
\begin{eqnarray*}
&& \frac{\sklfrac{1}{2}t(nr^d)^{-1}}{t^{-1}( \Ymax^2 +\eps^2)C_K\fmax
(n-1)r^d +\sklfrac{2}{3}(\Ymax+\eps)}
\\
&&\quad \to \frac{3\eps^2 \fmin^2
C_K}{6( \Ymax^2+\eps^2)\fmax+4\eps\fmin(\Ymax+\eps)}
\\
&&\quad > \frac{\eps^2 \fmin^2 C_K}{3( \Ymax^2+\eps^2)\fmax+2\eps\fmin
(\Ymax+\eps)}.
\end{eqnarray*}
Thus, for $n$ large enough we have
\[
\mathbb{P} \bigl(\hat f_n(X_1) \ge L\given
X_1=x, Y_1=y \bigr) \le e^{-C^\star_\eps nr^d},
\]
where
%
%
%e7.17 #&#
\begin{equation}
\label{eq:c_star_regression} C^\star_\eps= \frac{\eps^2 \fmin^2 C_K}{3( \Ymax^2+\eps^2)\fmax
+2\eps\fmin(\Ymax+\eps)}.
\end{equation}
Putting this back into \eqref{eq:regression_total_prob_lower}
completes the proof of \eqref{eq:prob_filtering_lower}.
The proof of \eqref{eq:prob_filtering_upper} is similar, with some
adjustments, and we omit it here.
\end{pf*}

%%%%%%%%%%

To prove Lemma~\ref{lem:set_bounds} we need the following lemma.

%%%%%%%%%%

%
%le7.4 #&#
\begin{lem} \label{lem:prob_coverage}
If $nr^d \to\infty$, then
\[
\mathbb{P} \bigl(D^{\downarrow}_{L+\eps}(2r) \not\subset\hat
D_L(n,r) \bigr) \le2 n e^{-C^\star_\eps n r^d},
\]
where $C^\star_\eps$ is the same as in Lemma~\textup{\ref{lem:prob_filtering}}.
\end{lem}

%%%%%%%%%%

\begin{pf}
Note that in both cases (density estimation and kernel regression) we
have that the set $D^{\downarrow}_{L+\eps}(2r)$ is bounded.
Let $\delta\in(0,1)$, and let $\cS\subset D^{\downarrow}_{L+\eps
}(2r)$ be a
finite set of points satisfying that for every $x\in D^{\downarrow
}_{L+\eps
}(2r)$ there exists $s\in\cS$ such that $\llVert  x-s\rrVert
\le\delta r$.
Then there exists a constant $c> 0$ such that we can construct $\cS$
with $\llvert {\cS}\rrvert  \le c (\delta r)^{-d}$ points.
Note that if there is
$x\in D^{\downarrow}_{L+\eps}(2r)$ that is not covered by the balls
of radius
$r$, it necessarily means that there is $s\in\cS$ that is not covered
by the balls of radius $(1-\delta)r$. Therefore,
\begin{eqnarray*}
\mathbb{P} \bigl(D^{\downarrow}_{L+\eps}(2r) \not\subset\hat
D_L(n,r) \bigr) &\le&\mathbb{P} \bigl(\exists s\in\cS:
B_{(1-\delta)r}(s) \cap\cX_n^L = \varnothing \bigr)
\\
& =&\mathbb{P} \bigl(\exists s\in\cS: B_{(1-\delta)r}(s) \cap
\cX_n^L = \varnothing; D^{\downarrow}_{L+\eps}(r)
\cap\cX_n\subset\cX _n^L \bigr)
\\
&&{} +\mathbb{P} \bigl(\exists s\in\cS: B_{(1-\delta)r}(s) \cap
\cX_n^L = \varnothing; D^{\downarrow}_{L+\eps}(r)
\cap\cX_n \not\subset\cX _n^L \bigr)
\\
&\le&\mathbb{P} \bigl(\exists s\in\cS: B_{(1-\delta)r}(s) \cap
\cX_n = \varnothing \bigr) +\mathbb{P} \bigl(D^{\downarrow}_{L+\eps}(r)
\cap\cX_n \not\subset \cX_n^L \bigr),
\end{eqnarray*}
where the last inequality is due to the fact that for every two events
$A,B$ we have $\mathbb{P} (A\cap B ) \le\mathbb{P} (A )$.
In other words the event of not covering $D^{\downarrow}_{L+\eps
}(2r)$ might
occur for two different reasons. Either the original sample (before
filtering) $\cX_n$ does not cover $D^{\downarrow}_{L+\eps}(2r)$ (the first
term), or our filtering method got rid of too many points (second term).
The second term can be bounded using Lemma~\ref{lem:prob_filtering}.
For the first term we have
\begin{eqnarray*}
\mathbb{P} \bigl(\exists s\in\cS: B_{(1-\delta)r}(s) \cap
\cX_n = \varnothing \bigr) &\le&\sum_{s\in\cS}\mathbb{P} \bigl(B_{(1-\delta)r}(s) \cap\cX_n = \varnothing \bigr)
\\
&=& \sum_{s\in\cS} \bigl(1- F\bigl(B_{(1-\delta)r}(s)
\bigr)\bigr)^{n}
\\
& \le&\sum_{s\in\cS} e^{-n F(B_{(1-\delta)r}(s))},
\end{eqnarray*}
where $F(A) = \int_A f_X(x)\,dx$. For the density estimation, $s\in
D^{\downarrow}
_{L+\eps}(2r)$ implies that
\[
F\bigl(B_{(1-\delta)r}(s)\bigr) \ge(L+\eps) (1-\delta)^d
\omega_dr^d \ge L(1-\delta)^d
\omega_d r^d.
\]
For the kernel regression model, we have that
\[
F\bigl(B_{(1-\delta)r}(s)\bigr) \ge\fmin(1-\delta)^d
\omega_dr^d.
\]

Thus, if we choose $C_1 = c\delta^{-d}$, and
\[
C_2 = \cases{ L(1-\delta)^d\omega_d, &\quad
density estimation,
\vspace*{3pt}\cr
\fmin(1-\delta)^d\omega_d, &\quad
kernel regression,}
\]
we have that
\[
\mathbb{P} \bigl(\exists s\in\cS: B_{(1-\delta)r}(s) \cap
\cX_n = \varnothing \bigr) \le C_1 r^{-d}
e^{-C_2 nr^d}.
\]
From Lemma~\ref{lem:prob_filtering} we know that
\[
\mathbb{P} \bigl(D^{\downarrow}_{L+\eps}(r)\cap
\cX_n \not\subset \cX_n^L \bigr) \le
ne^{-C^\star_\eps nr^d}.
\]
Note that for both models we have that $C^\star_\eps< C_2$ (see
\eqref{eq:c_star_density}, \eqref{eq:c_star_regression}), and also
that $r^{-d} = o(n)$. Therefore the latter probability is necessarily
the dominant one in the bound we have.
This completes the proof.
\end{pf}

%%%%%%%%%%

\begin{pf*}{Proof of Lemma~\ref{lem:set_bounds}}
If $nr^d\to\infty$, then by Lemma~\ref{lem:prob_coverage} we have
%
%
%e7.18 #&#
\begin{equation}
\label{eq:prob_down_bound} \mathbb{P} \bigl(D^{\downarrow}_{L+\eps}(2r)
\not\subset\hat D_L(n,r) \bigr) \le2n e^{-C^\star_\eps n r^d}.
\end{equation}
In addition, from Lemma~\ref{lem:prob_filtering} we have
%
%
%e7.19 #&#
\begin{equation}
\label{eq:prob_up_bound} \mathbb{P} \bigl(\hat D_L(n,r) \not\subset
D^{\uparrow}_{L-\eps
}(2r) \bigr) \le\mathbb{P} \bigl(
\cX_n^L \cap\bigl(D^{\uparrow}_{L-\eps}(r)
\bigr)^c\ne \varnothing \bigr) \le ne^{-C^\star_\eps n r^d}.
\end{equation}
Using the union bound completes the proof.
\end{pf*}

%%%%%%%%%%

The last piece of the puzzle is the proof of the algebraic Lemma~\ref
{lem:algebra}.

\begin{pf*}{Proof of Lemma~\ref{lem:algebra}}
We need to show that $g_{{34}}$ is injective and that $\im(g_{{34}})
= G_{24}$.
\begin{enumerate}[2.]
\item The assumption that $g_{{35}}$ is an isomorphism from $G_3$ to
$G_{15}$ implies that $g_{{35}}$ is injective. Since $g_{{35}} =
g_{{45}}\circ g_{{34}}$ we have that $g_{{34}}$ is injective as well.

\item Since (a) $g_{{15}}:G_1\to G_{15}$ is surjective, (b)
$g_{{35}}:G_3\to G_{15}$ is an isomorphism, and (c) $g_{{15}} =
g_{{35}}\circ g_{{13}}$, we conclude that $g_{{13}}:G_1\to G_3$ is
surjective. Since $g_{{13}} = g_{{23}}\circ g_{{12}}$, we have that
$g_{{23}}$ is surjective as well.

Finally, since (a) $\im(g_{{24}}) = G_{24}$, (b) $g_{{24}} =
g_{{34}}\circ g_{{23}}$, and (c) $g_{{23}}:G_2\to G_3$ is
surjective, we have that $\im(g_{{34}}) = G_{24}$ as well.
\end{enumerate}
\end{pf*}
\end{appendix}

% zodis "Acknowledgments" paliekamas pagal autoriu
\section*{Acknowledgements}
The authors would like to thank: Robert Adler, Paul Bendich, Ulrich
Bauer, Ezra Miller and Andrew Nobel for many useful discussions.
We also wish to thank Fr\'{e}d\'{e}ric Chazal, Larry Wasserman and the
anonymous referees for very useful comments on previous revisions of
this paper.

Omer Bobrowski gratefully acknowledges the support of AFOSR:
FA9550-10-1-0436, and NSF DMS-1127914.
Sayan Mukherjee is pleased to acknowledge support from grants AFOSR:
FA9550-10-1-0436, and NSF: CCF-1049290.
Jonathan E. Taylor was supported by the AFOSR, Grant 113039.

%\begin{supplement}%[id=suppA]
%\sname{Supplement A}
%\stitle{}
%\slink[doi]{10.3150/00-BEJXXXXSUPP} %[doi,text={...}] - jei reikia
%suskaldyti doi
%\sdatatype{.pdf}
%\sfilename{BEJ000\_supp.pdf}
%\sdescription{}
%\end{supplement}

% imsref loaded by linak, 2015-10-22 11:14:46
%

\printhistory
\end{document}